\documentclass[12pt]{amsart}

\usepackage{amsfonts,amsmath,amssymb,amsthm}
\usepackage{hyperref}
\hypersetup{breaklinks=true}

\numberwithin{equation}{section}
\theoremstyle{plain}
\newtheorem{theorem}[equation]{Theorem}  
\newtheorem{lemma}[equation]{Lemma}       
\newtheorem{proposition}[equation]{Proposition}      
\newtheorem{corollary}[equation]{Corollary}

\newtheorem{sublemma}[equation]{Sublemma}

\theoremstyle{remark}

\newtheorem*{remark*}{Remarks}  
\newtheorem{example}[equation]{Example}
\newtheorem*{example*}{Example}

\theoremstyle{definition}
\newtheorem{definition}[equation]{Definition}

\newcommand{\al}{\alpha}

\newcommand{\cont}{\operatorname{cont}}
\newcommand{\C}{{\mathcal C}}
\newcommand{\Confdim}{\mathop{\mathrm{Confdim}}\nolimits}

\newcommand{\D}{\partial}
\newcommand{\De}{\Delta}

\newcommand{\diam}{\mathop{\mathrm{diam}}\nolimits}
\newcommand{\dist}{\mathop{\mathrm{dist}}\nolimits}
\newcommand{\inte}{\mathop{\mathrm{int}}\nolimits}
\newcommand{\eps}{\epsilon}

\newcommand{\Ga}{\Gamma}
\newcommand{\ga}{\gamma}

\newcommand{\hr}{\mathbb{R}}
\newcommand{\hh}{\mathbb{H}}
\newcommand{\hn}{\mathbb{N}}

\newcommand{\Mod}{\mathop{\mathrm{Mod}}\nolimits}

\newcommand{\ra}{\rightarrow}
\newcommand{\R}{\mathbb{R}}
\newcommand{\restr}{\mbox{\Large \(|\)\normalsize}}
\newcommand{\sep}{\operatorname{sep}}

\renewcommand{\S}{\mathbb{S}}

\newcommand{\T}{\mathcal{T}}
\newcommand{\val}{\mathop{\mathrm{val}}\nolimits}
\newcommand{\Z}{\mathbb{Z}}

\newcommand{\la}{\lambda}
\newcommand{\spt}{\operatorname{Spt}}

\newcommand{\Lip}{\operatorname{Lip}}

\begin{document}

\title[Some applications of $\ell_p$-cohomology]{Some applications of $\ell_p$-cohomology to boundaries of Gromov
hyperbolic spaces}

\author{ Marc Bourdon and Bruce Kleiner} 
\maketitle

\begin{abstract}
We study quasi-isometry invariants of Gromov hyperbolic spaces, focussing  on the 
$\ell_p$-cohomology and closely related invariants such as the conformal dimension,
combinatorial modulus, and the Combinatorial Loewner Property.    We give  new
constructions of continuous $\ell_p$-cohomology, thereby obtaining
information about the $\ell_p$-equivalence relation, as well
as critical exponents associated
with $\ell_p$-cohomology.   As an application, we provide a flexible construction
of hyperbolic groups which do not have the Combinatorial Loewner Property, 
extending \cite{B2} and
complementing the examples from \cite{BourK}.  Another consequence is the existence of
hyperbolic groups with  Sierpinski carpet boundary which have conformal dimension 
arbitrarily close to $1$.  In particular, we
answer questions of Mario Bonk,  Juha Heinonen and John Mackay.
\end{abstract}


\setlength{\parskip}{\smallskipamount}

\tableofcontents

\setlength{\parskip}{\medskipamount}

\section{Introduction}
\subsection*{Background}
In this paper we will be interested in quasi-isometry invariant structure in Gromov
hyperbolic spaces, primarily  structure which is reflected in the boundary.   For some
hyperbolic groups $\Gamma$, the topological
structure of the boundary $\D \Gamma$ alone contains substantial information:
witness the JSJ decomposition encoded in the local cut point structure of the boundary
\cite{bowditchlocalcutpoints}, and many  situations where one can detect 
boundaries of certain
subgroups $H\subset \Gamma$ by means of topological criteria. 
However, in many cases, for instance for generic hyperbolic groups, the topology 
reveals little of the structure of the group and is completely inadequate for 
addressing rigidity questions, since the homeomorphism group
of the boundary is highly transitive.    In these cases it is necessary to use the 
finer quasi-M\"obius structure of the boundary and analytical invariants attached to it,
such as  modulus (Pansu, metric-measure, or combinatorial), $\ell_p$-cohomology,
and closely related quantities such as the conformal dimension.  The seminal
work of Heinonen-Koskela \cite{HeK},  followed by 
\cite{cheeger,khshanmuga,Ty,keith} indicates that when 
$\D \Gamma$ is quasi-M\"obius homeomorphic to a Loewner space
(an Ahlfors $Q$-regular $Q$-Loewner space in the sense of \cite{HeK}),
there should be a rigidity theory resembling that of lattices in rank
$1$ Lie groups.   This possibility is illustrated by \cite{BP1,xie}. 

It remains unclear which hyperbolic groups $\Gamma$ have Loewner boundary
in the above sense.
Conjecturally, $\D \Gamma$ is Loewner 
if and only if it satisfies the Combinatorial Loewner
Property \cite{K}.   To provide some evidence of the abundance of such groups
(modulo the conjecture),
in \cite{BourK} we gave a variety of examples with the Combinatorial Loewner Property.
On the other hand, it had already been shown in \cite{BP2,B2} that there are
groups $\Gamma$ whose boundary is not Loewner, which can still be 
effectively studied using  $\ell_p$-cohomology and its cousins.   
Our main purpose in this paper  is to advance the understanding of this 
complementary situation  by providing new constructions of $\ell_p$-cohomology
classes, and giving a number of applications.

\subsection*{Setup}
We will restrict our attention (at least in the introduction) to proper
Gromov hyperbolic spaces which satisfy the following two additional conditions:
\begin{itemize}
\item (Bounded geometry) \; For every $0<r \le R$, there is 
a $N(r,R) \in \mathbb N$ such that every
$R$-ball can be covered by at most $N=N(r,R)$ balls
of radius $r$.
\item (Nondegeneracy) \;There is a $C\in [0,\infty)$ such that every point $x$ lies within 
distance at most $C$ from  all three sides of some ideal geodesic triangle $\De_x$.
\end{itemize}
The visual boundary $\partial X$ of such a space $X$
is a compact, doubling, uniformly perfect metric space, which is
determined 
up to  quasi-M\"obius homeomorphism by the quasi-isometry class of $X$.  
Conversely, every  compact, doubling, uniformly perfect metric space 
is the visual boundary of  a unique hyperbolic metric space as above, 
up to quasi-isometry (see Section \ref{preliminaries}).  

To simplify the discussion of homological properties,   
we will impose (without
loss of generality) the additional standing assumption
 that $X$ is a simply connected metric simplicial complex with links
of uniformly bounded complexity, and with all simplices isometric 
to regular Euclidean simplices with unit length edges.

\subsection*{Quasi-isometry invariant function spaces}
Let $X$ be a Gromov hyperbolic simplicial complex as above, with 
boundary $\D X$.   

We recall  (see Section \ref{equivalence} and \cite{Pa1,G,E,B2}) that 
for $p\in (1,\infty)$,  the  continuous (first) 
$\ell_p$-cohomology $\ell_pH^1_{\cont}(X)$  
is canonically isomorphic 
to the space  $A_p(\D X)/ \mathbb R$, 
where $A_p(\D X)$ is the space of 
continuous functions 
$u:\D X\ra \R$ that have a continuous extension
$f:X^{(0)}\cup \D X\ra \R$ with  $p$-summable
coboundary:
$$
\|d f\|^p_{\ell_p}=\sum_{[vw]\in X^{(1)}} |f(v)-f(w)|^p\; <\infty\,,
$$
and where $\mathbb R$ denotes the subspace of constant functions.
Associated with the continuous $\ell_p$-cohomology are several other
quasi-isometry invariants:
\begin{enumerate}
\item The $\ell_p$-equivalence relation $\sim_p$ on $\D X$, where $z_1\sim_p z_2$
iff $u(z_1)=u(z_2)$ for every $u\in A_p(\D X)$.    
\item The infimal  $p$ such that $\ell_pH^1_{\cont}(X)\simeq A_p(\D X)/ \mathbb R$
is nontrivial.   We will denote this by $p_{\neq 0}(X)$.   Equivalently $p_{\neq 0}(X)$ is the infimal $p$
such that $\sim_p$ has more than one coset.
\item  The infimum $p_{\sep}(X)$ of the $p$ such that 
$A_p(\D X)$ separates points in $\D X$, or equivalently, $p_{\sep}(X)$ is
the infimal $p$ such that all
cosets of $\sim_p$ are points.
\end{enumerate}
These invariants were exploited in 
\cite{BP2,B2,BourK} due to their connection with conformal dimension
and the Combinatorial Loewner Property.  Specifically, when 
$\D X$ is approximately self-similar (e.g.\! if $\D X$
is the visual boundary of a hyperbolic group)  then  $p_{\sep}(X)$
coincides with the Ahlfors regular conformal dimension 
of $\D X$; and if $\D X$ has the Combinatorial Loewner Property then the two 
critical exponents  $p_{\neq 0}(X)$ and $p_{\sep}(X)$ coincide, i.e.\! for every $p\in (1,\infty)$, the function space
$A_p(\D X)$ separates points iff it is nontrivial  (we refer the reader
to Sections \ref{preliminaries}--\ref{cohoboundary} for the precise statements and
the relevant terminology).

\subsection*{Construction of nontrivial continuous $\ell_p$-cohomology}
The key results in this paper are constructions of nontrivial elements in the $\ell_p$-cohomology.
The general approach for the construction is inspired by 
\cite{B2}, and may be described as follows.  
Inside the Gromov hyperbolic complex
$X$, we identify a subcomplex $Y$ such that the relative
cohomology of the pair $(X, X\setminus Y)$ reduces -- essentially by excision -- to 
the cohomology of $Y$ relative to its frontier in $X$.    Then we prove that the latter
contains an abundance of nontrivial classes.   This yields nontrivial classes in 
$\ell_pH^1_{\cont}(X)$ with additional control, allowing us to make deductions about
the cosets of the $\ell_p$-equivalence relation.

Let $Y$  be a Gromov hyperbolic space satisfying our standing assumptions.
We recall \cite{E} that if  $Y$
has Ahlfors $Q$-regular visual boundary $\D Y$, then $A_p(\D Y)$ contains the 
Lipschitz functions on $\D Y$ for any $p>Q$, and in particular it separates points.   
Our first result says that if $W\subset Y$ is a subcomplex with well-separated connected 
components, then for $p$ slightly larger than $Q$,
the relative cohomology $\ell_pH^1_{\cont}(Y,W)$ is highly 
nontrivial.   In other words, at the price of increasing the exponent slightly, one
can arrange for the representing functions $f:Y^{(0)}\ra \R$ to be constant
on the ($0$-skeleton of the) connected components of $W$, provided the 
components are far apart. 

\begin{theorem}[Corollary \ref{corqualitative}]\label{thm-qualitative_hyperbolic}
Let $Y$ be hyperbolic simply connected metric simplicial complex satisfying the  assumptions above,
and let $\Confdim (\partial Y)$
denote the Ahlfors regular conformal dimension of $\partial Y$.

For every $\alpha \in (0,1)$, $C \ge 0$, there is a $D \ge 1$
with the following properties.
Suppose $W\subset Y$ is a subcomplex of $Y$
such that
\begin{enumerate}
\item Every connected component of $W$ is  $C$-quasiconvex in $Y$.
\item For every $y \in W$ there is a
complete geodesic $\ga\subset W$ lying in the same connected component of 
$W$, such that $\dist(y, \ga) \le C$.
\item The distance between distinct components of $W$ is at least $D$.
\end{enumerate} 
Then for  $p > \frac{1}{\alpha} \Confdim (\partial Y)$,  the relative cohomology 
$\ell_pH^1_{\cont}(Y,W)$ separates every pair of distinct points $z_1, z_2 \in \partial Y$,
such that $\{z_1, z_2 \} \nsubseteq \partial H$ for every connected component $H$ of $W$.
\end{theorem}

We prove Theorem \ref{thm-qualitative_hyperbolic} by translating it to 
an existence theorem for functions on $\D Y$.   It then reduces to the 
following
theorem about H\"older functions, which is of  independent interest:

\begin{theorem}[Theorem \ref{thm-holderfunction}]
\label{thm-intro_holderfunction}
For every $\al\in (0,1)$, there is a $D \ge 1$ with the following property.  
Suppose
$Z$ is a  bounded metric space, 
and $\C$ is a countable collection of closed positive diameter subsets of $Z$   
such that
the pairwise relative distance satisfies 
$$
\Delta(C_1,C_2) =\frac{d(C_1,C_2)}{\min(\diam(C_1),\diam(C_2))}\ge D
$$
for all $C_1,C_2\in \C$, $C_1\neq C_2$.   Then for every pair of distinct
points $z_1, z_2 \in Z$, 
either  $\{z_1, z_2 \} \subseteq C$ for some $C \in \mathcal C$, or
there is a H\"older function $u\in C^\al(Z)$, 
such that: 
\begin{enumerate}
\item  $u\restr_C$ is 
constant for every $C\in \C$.
\item If $C_1,C_2\in\C$ and  $u(C_1)= u(C_2)$, then $C_1=C_2$.
\item $u(z_1) \neq u(z_2)$. 
\end{enumerate}
\end{theorem} 

We remark that one can construct an example $Z$, $\C$ as in the
theorem, where $Z=S^k$ and $\cup_{C\in\C}\,C$ has full 
Lebesgue measure.
In this case, if $u:Z\ra \R$ is a  
Lipschitz function that is constant on every $C\in\C$, then  almost every point $z\in Z$
will be a point of differentiability of $u$ and a point of density of some element $C\in \C$.
Hence $Du=0$ almost everywhere, and $u$ is constant.
This shows that it is necessary to take $\al<1$.

\bigskip
Our second construction of $\ell_p$-cohomology classes pertains to a special class of 
$2$-complexes.

\begin{definition} An \emph{elementary polygonal complex} 
is a connected,
simply connected, $2$-dimensional cell
complex $Y$ whose edges are colored black or white,  that enjoys the following properties:
\begin{itemize}
\item $Y$ is a union of $2$-cells that intersect pairwise in at most a single vertex or edge.
\item Every $2$-cell is combinatorially isomorphic to a polygon with even perimeter 
at least $6$.
\item The edges on the boundary 
of every $2$-cell are alternately black and white. 
\item Every white edge has thickness one, and every black
edge has thickness at least $2$.   Here the thickness of an edge is the number
of $2$-cells containing it.
\end{itemize}
The union of the white edges in $Y$ is the {\em frontier} of $Y$, and is denoted $\T$;
its connected components are trees.
\end{definition}

A simple example of an elementary polygonal complex is the orbit $\Ga(P)$ of a right-angled
hexagon $P$ in the hyperbolic plane, under the group $\Ga $ generated by
reflections in $3$ alternate sides $e_1,e_2,e_3\subset \D P$.

Elementary polygonal complexes are relevant for us because
they turn up naturally as embedded subcomplexes $Y\subset X$, where $X$ is a  generic polygonal 
$2$-complex, such that the pair $(Y,\T)$ is the result of applying excision to 
the pair $(X,X\setminus Y)$.  

For every   elementary polygonal $2$-complex $Y$, we define two invariants:
\begin{itemize}
\item $p _{\neq 0}(Y,\T)$ is  the infimum of the set of $p\in [1,\infty)$ such that 
the relative cohomology $\ell_{p}H^1_{\cont}(Y,\T)$
is nontrivial.
\item $p _{\textrm{sep}} (Y,\T)$ is the infimum of the set of  $p\in [1,\infty)$, such that
for every pair of distinct points $z_1, z_2 \in \partial Y$,
either $\{z_1, z_2 \} \subseteq \partial T$ for some connected component $T$ of $\T$,
or the relative cohomology $\ell_{p}H^1_{\cont}(Y,\T)$ separates 
$\{z_1,z_2\}$.
\end{itemize}
Our main result about elementary polygonal complexes
is the following pair of estimates on these invariants:

\begin{theorem}[Corollary \ref{corpq}] \label{thm-pq_estimate} Let $Y$ be an elementary polygonal complex. 
Assume that the perimeter of every $2$-cell of $Y$ lies in $[2m_1, 2m_2]$
and that the thickness of every black edge lies in $[k_1, k_2]$,
with $m_1 \ge 3$ and $k_1 \ge 2$. Then: \;
$$1 + \frac{\log (k_1-1)}{\log (m_2-1)}\le p_{\neq 0}(Y, \mathcal T) \le p_{\textrm{sep}}(Y, \mathcal T) 
\le 1 + \frac{\log (k_2-1)}{\log (m_1-1)}.$$
\end{theorem}
Note that the estimate becomes sharp when the thickness and perimeter are
constant.

\subsection*{Applications to spaces and groups}
We now discuss some of applications given in the paper.

We first apply Theorem \ref{thm-qualitative_hyperbolic}
to amalgams, generalizing the construction of
\cite{B2}.

\begin{theorem}[Corollary \ref{cor2amalgam}]\label{thm-intro_amalgam}
Suppose $A,B$ are hyperbolic groups, and we are given
malnormal quasiconvex embeddings $C\hookrightarrow A$, $C\hookrightarrow B$.  Suppose 
that there is a decreasing sequence 
$\{A_n\}_{n \in \mathbb N}$ of finite index subgroups of $A$ such that 
$\cap _{n \in \mathbb N} A_n = C$, and
set $\Gamma \!_n := A_n \star _C B$.
\begin{itemize}\item [(1)]  If $p_{\textrm{sep}} (A) < p_{\textrm{sep}} (B)$ then, 
for all $p \in (p_{\textrm{sep}} (A), p_{\textrm{sep}} (B)]$
and every $n$ large enough, the $\ell _p$-equivalence relation on $\partial \Gamma \!_n$ 
possesses a coset different from a point
and the whole $\partial \Gamma$. In particular for $n$ large enough, 
$\partial \Gamma \!_n$ does not have the CLP.
\item [(2)] If $p_{\textrm{sep}} (A) < p_{\neq 0}(B)$ then, for $p \in (p_{\textrm{sep}}(A), p_{\neq 0}(B))$
and every $n$ large enough, the cosets of the $\ell_p$-equivalence relation on $\partial \Gamma \!_n$
are single points and the boundaries of cosets $gB$, for $g \in \Gamma \!_n$.
In particular, for large $n$, any quasi-isometry of $\Gamma \!_n$ permutes the cosets $gB$, for
$g\in \Gamma \!_n$. 
\end{itemize}
\end{theorem}
We now sketch the proof of the theorem.

Let $K_A$, $K_B$, and $K_C$ be finite $2$-complexes with
respective fundamental groups $A$, $B$, and $C$, such that there are simplicial
embeddings $K_C\hookrightarrow K_A$, $K_C\hookrightarrow K_B$ inducing the
given embeddings of fundamental groups.    For every $n$, let $K^n_A\ra K_A$ be the 
finite covering corresponding to the inclusion 
$A_n\subset A$, and fix a lift $K_C\hookrightarrow K^n_A$
of the embedding $K_C\hookrightarrow K_A$.     We let $K^n$ be the result of 
gluing $K^n_A$ to $K_B$ along the copies of $K_C$, so 
$\pi_1(K^n)\simeq A_n\star_C B=\Gamma_{\!n}$.   The universal cover $\tilde K^n$
is a union of copies of the universal covers
$\tilde K^n_A=\tilde  K_A$ and $\tilde K_B$, whose incidence
graph is the Bass-Serre tree of the decomposition $\Gamma_{\!n}=A_n\star_C B$.
If we choose a copy $Y\subset \tilde K^n$ of $\tilde K_A$, then the frontier $W_n$ of $Y$
in $\tilde K^n$ breaks up into connected components which are stabilized by  conjugates
of $C$,
where the minimal pairwise separation between distinct
components of $W_n$ tends to infinity as $n\ra \infty$.   
Theorem \ref{thm-qualitative_hyperbolic} then applies,
yielding
nontrivial functions in  $\ell_pH^1_{\cont}(Y,W_n)$
for every $p>\Confdim(A)$ and every $n$ sufficiently large; 
by excision these give functions in  $\ell_pH^1_{\cont}(\tilde K^n)\simeq 
A_p(\D \Gamma\! _n)/ \mathbb R$ which are constant
on $\D Z\subset\D \Gamma\!_n$ for every copy $Z\subset \tilde K^n$
of $\tilde K_B$ in $\tilde K^n$.   These functions provide enough information
about the $\ell_p$-equivalence relation $\sim_p$  to deduce (1) and (2).

As a corollary, we obtain examples of hyperbolic groups with Sierpinski carpet boundary
which do not have the Combinatorial Loewner Property, and which are quasi-isometrically
rigid.  See Example \ref{ex-marios_question}.  These examples answer a question of Mario Bonk.

Using our estimates for the critical exponents for elementary polygonal complexes, we
obtain  upper bounds for the Ahlfors regular conformal dimension of boundaries of 
a large class of $2$-complexes.
\begin{theorem}[Proposition \ref{propappli}]\label{thm-intro_polygonal_complex}
Let $X$ be a simply connected hyperbolic $2$-complex whose
boundary is connected and approximately self-similar.
Assume that $X$ is a union of $2$-cells,
where  $2$-cells intersect pairwise in at most a vertex or edge.
\begin{enumerate}
\item If the perimeter of every $2$-cell is at least $n \ge 5$,
the thickness of every edge lies is at most $k \ge 2$, and the link of every vertex contains
no circuit of  length $3$, then 
$$\Confdim (\partial X) \le 1 + \frac{\log (k-1)}{\log (n-3)}.$$
\item If the perimeter of every $2$-cell is at least $n \ge 7$ and
the thickness of every edge is at most $k \ge 2$, then
$$\Confdim (\partial X) \le 1 + \frac{\log (k-1)}{\log (n-5)}.$$
\end{enumerate} 
\end{theorem}

To prove Theorem \ref{thm-intro_polygonal_complex}(1), by a 
straightforward consequence of \cite{KeK, Car},
it suffices to  show that the function space $A_p(\D X)$ separates points in $\D X$ for $p>
1 + \frac{\log (k-1)}{\log (n-3)}$.  To do this, we find that (a subdivision of)
$X$ contains many elementary polygonal complexes $Y$ with thickness at most $ k$
and perimeter at least $2(n-2)$, such that excision applied to $(X,X\setminus Y)$
yields the pair $(Y,\T)$.     Theorem \ref{thm-pq_estimate}
then produces enough elements in $A_p(\D X)$ to separate points.

As an application 
of  Theorem \ref{thm-intro_polygonal_complex} we obtain, for every $\eps>0$,
a hyperbolic  group $\Gamma$
with Sierpinski carpet boundary 
with the Combinatorial Loewner Property,
such that $\Confdim(\D \Gamma)<1+\eps$.   See Example \ref{ex-johns_question}.
This answers a question of  Juha Heinonen  and John
Mackay.

\subsection*{Organization of the paper}
Section \ref{preliminaries} is a brief presentation of some useful topics in geometric analysis.
Section \ref{equivalence} discusses first $\ell_p$-cohomology and related invariants.
We present our new construction of continuous $\ell_p$-cohomology, and discuss several 
connections between $\ell_p$-cohomology and the geometry of the boundary.
In Section \ref{qualitative} we establish Theorems \ref{thm-qualitative_hyperbolic} and 
\ref{thm-intro_holderfunction} .
Applications to amalgamated products are given in Section \ref{amalgam}.
Section \ref{elementary} focusses on elementary polygonal complexes. 
Applications to polygonal complexes and
Coxeter groups are given in Sections \ref{polygonal} and \ref{sectioncox}.

\subsection*{Acknowledgements} We are grateful to John Mackay for his comments on an earlier
version of the paper.

\section{Preliminaries} 
\label{preliminaries}
This section is a brief presentation of some topics in geometric analysis
that will be useful in the sequel.
This includes the Ahlfors regular conformal dimension, the Combinatorial Loewner Property,
approximately self-similar spaces, and a few aspects of Gromov hyperbolic spaces.

\subsection*{Ahlfors regular conformal dimension} 
We refer to \cite{MT} for a detailed treatment of the conformal dimension and related subjects.

Recall that a metric space $Z$ is called a \emph{doubling metric space} if 
there is a constant $n\in\mathbb{N}$ such that every ball $B(z,r) \subset Z$ can be covered by
at most $n$ balls of radius $\frac{r}{2}$. 

A space $Z$  is \emph{uniformly perfect },
if there exists a constant
$0< \lambda <1$ such that for every ball $B(z, r)$ in $Z$ with $0 < r \le \diam Z$ one has
$ B(z,r) \setminus B(z, \lambda r) \neq \emptyset$. 

A space $Z$ is  \emph{Ahlfors $Q$-regular} (for some $Q \in (0, +\infty)$) if there
is a measure $\mu$ on $Z$ such that for every ball $B \subset Z$ of radius 
$0<r \le \diam(Z)$ one has $\mu(B) \asymp r^Q$. 

Every compact, doubling, uniformly perfect metric space is quasi-M\"obius homeomorphic
to a Ahlfors regular metric space (see \cite{He}). 
This justifies the following definition.

\begin{definition} Let $Z$ be a compact, doubling, uniformly perfect metric space.
The \emph{Ahlfors regular conformal dimension} of $Z$ 
is the infimum of the Hausdorff dimensions of the Ahlfors regular
metric spaces which are quasi-M\"obius homeomorphic to $Z$. We shall denote it
by $\Confdim (Z)$.
\end{definition}

\subsection*{The Combinatorial Loewner Property (CLP)}
The Combinatorial Loewner Property was introduced in \cite{K, BourK}.
We start with some basic related notions. 

Let $Z$ be a compact metric space, let $k \in \hn$, and let 
$\kappa \geq 1$. A finite graph
$G_k$ is called a \emph{$\kappa$-approximation of $Z$ on scale $k$},
 if it is
the incidence graph of a
covering of $Z$,  
such that for every $ v \in G_k ^0$ there
exists $z_v \in Z$ with
$$ B(z_v, \kappa ^{-1} 2^{-k}) \subset v \subset B(z_v, \kappa 2^{-k}), $$
and for $v, w \in G_k ^0$ with  $v \neq w $:
$$  B(z_v, \kappa ^{-1} 2^{-k}) \cap B(z_w, \kappa ^{-1} 2^{-k}) = 
\emptyset. $$
Note that we identify every vertex $v$ of $G_k$ with the 
corresponding subset in $Z$.
A collection of graphs $\{ G_k \} _{k \in \hn}$ is called a
\emph{$\kappa$-approximation} of $Z$, if for each $k \in \hn$
the graph $G_k$ is a
$\kappa$-approximation of $Z$ on scale $k$.

Let $\gamma \subset Z$ be a  curve and let $\rho :
G_k ^0 \to \hr_+ $ be any function. The \emph{$\rho$-length} of $\gamma$ is
$$ L_{\rho} (\gamma) = \sum_{v \cap \gamma \neq \emptyset} \rho (v).$$
For $p \geq 1$ the \emph{$p$-mass} of $\rho$ is
$$ M_p (\rho) = \sum_{v \in G_k ^0} \rho (v)^p.$$
Let $\mathcal{F}$ be a non-void family of  curves in $Z$.   We define
the \emph{$G_k$-combinatorial $p$-modulus} by
$$\Mod _p (\mathcal{F}, G_k) = \inf _{\rho} M_p (\rho),$$
where the infimum is over all \emph{$\mathcal{F}$-admissible functions}, 
i.e.\! functions  $\rho : G_k ^0 \to \hr_+ $ which satisfy 
$ L_{\rho} (\gamma) \geq 1$ for every $\gamma \in \mathcal{F}$.

We denote  by $\mathcal F (A, B)$ the family of curves joining 
two subsets $A$ and
$B$ of $Z$ and by $\Mod _p (A, B, G_k)$ its  $G_k$-combinatorial
$p$-modulus.

Suppose now that $Z$ is a compact arcwise connected doubling metric
space. Let 
$\{ G_k \} _{k \in \hn}$ be a $\kappa$-approximation of $Z$.
In the following statement $\Delta (A,B)$ denotes the \emph{relative distance} between two disjoint 
non degenerate continua $A, B \subset Z$ i.e.\!
\begin{equation}\label{reldist}
\Delta (A, B) = \frac {\dist (A, B) } {\min \{\diam A, \diam B \}}.
\end{equation}

\begin{definition}
\label{def-clp}
Suppose  $p >1$.   Then $Z$ satisfies the \emph{Combinatorial
$p$-Loewner Property} if there exist 
two positive increasing functions  $\phi , \psi $ on $(0, +\infty)$ 
with $\lim _{t \to 0} \psi (t) = 0$, such that for 
all disjoint non-degenerate continua $A, B \subset Z$
and for all $k$ with $2^{-k} \le \min \{\diam A, \diam B \} $ 
one has:
$$ \phi (\Delta (A, B) ^{-1}) \le \Mod _p (A, B, G_k)
\le \psi (\Delta (A, B) ^{-1}) .$$
We say that $Z$ satisfies the \emph{Combinatorial
Loewner Property} if it satisfies the Combinatorial
$p$-Loewner Property for some $p>1$.
\end{definition}

The CLP is invariant under quasi-M\"obius homeomorphisms. 
A compact Ahlfors $p$-regular, $p$-Loewner metric space
satisfies the Combinatorial $p$-Loewner Property (see \cite{BourK} Th.2.6 for these results). 
It is conjectured in \cite{K} that if $Z$ satisfies the CLP and is approximately self-similar
(see below for the definition), then $Z$ is quasi-M\"obius homeomorphic to a regular Loewner 
space. 

\subsection*{Approximately self-similar spaces}

The following definition appears in \cite{K}, \cite{BourK}.

\begin{definition} \label {homothety} A compact metric space $(Z, d)$ is called \emph{approximately
self-similar} if there is a constant $L_0 \ge 1$ such that if $B(z,r) \subset Z$ is a ball 
of radius $0 < r \le \diam (Z)$, then there is an open subset $U \subset Z$ which is $L_0$-bi-Lipschitz
homeomorphic to the rescaled ball $(B(z,r), \frac{1}{r} d)$.
\end{definition}

Observe that approximately
self-similar metric spaces are doubling and uniformly perfect. 
Examples include some classical fractal spaces like the square Sierpinski
carpet and the cubical  Menger sponge. Other examples are the visual boundaries of
the hyperbolic spaces
which admit an isometric properly discontinuous and cocompact group action
\cite{BourK}.
A further source of examples comes from expanding Thurston  maps,
\cite{bonkmeyer}, \cite{H3}. 

The following result is due to S. Keith and the second (named) author \cite{KeK}.
A proof is written in \cite{Car}.

\begin{theorem}\label{modconf} Suppose $Z$ is an arcwise connected,
approximately self-similar metric space.  Let $\{ G_k \} _{k \in \hn}$ 
be a $\kappa$-approximation of $Z$. Pick a positive constant $d_0$ that is small compared to 
the diameter of $Z$  and to the constant $L_0$ of Definition \ref{homothety}. 
Denote by $\mathcal F _0$ the family of curves $\gamma \subset Z$ with $\diam
(\gamma) \ge d_0$. 
Then 
$$\Confdim (Z) = \inf \{p \in [1, +\infty) \;|
\lim _{k \to +\infty} \Mod _p (\mathcal{F} _0 , G_k) = 0 \}.$$
\end{theorem}

\subsection*{Hyperbolic spaces} 
Let  $X$ be a hyperbolic proper
geodesic metric space. We denote the distance between any pair of points 
$x, x' \in X$ by $\vert x - x' \vert$.
%
Let $\partial X$ be the boundary at infinity of $X$. It carries a \emph{visual metric},
i.e.\! a metric $d$ for which there is a  constant $a >1$ such that
for all $z, z' \in \partial X$, one has
\begin{equation} \label{visual}
d(z,z') \asymp a ^{- L},
\end{equation}
where $L$ denotes the distance from $x_0$ (an origin in $X$) to 
a geodesic $(z, z') \subset X$.
Moreover $X \cup \partial X$ is naturally a metric compactification of $X$.
 There is a metric $d$ on $X \cup \partial X$ that enjoys the following property: for all
$x, x' \in X \cup \partial X$, one has
\begin{equation} \label{visualbis}
d(x, x') \asymp a^{-L} \min \{1, \vert x - x'
\vert \},
\end{equation}
where $a$ is the exponential parameter in (\ref{visual}), and where $L$ denotes the distance
in $X$ from the origin $x_0$ to a geodesic whose endpoints are $x$ and $x'$.
See e.g.\! \cite{GO, BHK} for more details. For a subset $E \subset X$
we denote by $\partial E$ its \emph{limit set} in $\partial X$, i.e.\!
$\partial E = \overline{E}^{X \cup \partial X} \cap \partial X$.

If $X$ satisfies the bounded geometry and nondegeneracy conditions of the introduction,
then $\partial X$ is a doubling uniformly perfect metric space. 
Conversely, every  compact, doubling, uniformly perfect metric space 
is the visual boundary of  a unique hyperbolic metric space as above, 
up to quasi-isometry ; this follows e.g.\! from a contruction of G. Elek
\cite{E} (see also \cite{BP3} for more details).

We also notice that -- thanks to Rips' construction (see \cite{GO, BH, KapB}) -- 
every proper bounded geometry hyperbolic space is 
quasi-isometric to a contractible simplicial metric complex,
with links of uniformly bounded complexity, and all simplices isometric to regular
Euclidean simplices with unit length edges.

\section{$\ell_p$-cohomology} \label{equivalence}
This section presents aspects of the $\ell_p$-cohomology that will be useful in the 
sequel. Only the first $\ell_p$-cohomology will play a role. 
 Therefore, instead of considering contractible simplicial complexes,
we will content ourselves with simply connected ones.

We consider in this section 
a connected, simply connected, metric simplicial complex $X$, with links of uniformly
bounded complexity, and all simplices isometric to regular
Euclidean simplices with unit length edges.  We suppose that 
it is a hyperbolic metric space, and we
denote its visual boundary by $\partial X$.

\subsection*{First $\ell_p$-cohomology}

For $k \in \mathbb N$ denote by $X^{(k)}$ the set of the $k$-simplices of $X$.
For a countable set $E$ and for $p\in [1,\infty)$, let $\ell_p(E)$
be the Banach space of $p$-summable real functions on $E$. The
\emph{$k$-th space of 
$\ell_p$-cochains} is $C_p ^{(k)} (X) := \ell_p(X^{(k)})$. The standard coboundary
operator 
$$d^{(k)} : C_p ^{(k)} (X) \to C_p ^{(k+1)} (X)$$
is bounded because of the bounded geometry assumption on $X$. 
When $k=0$, the operator
$d^{(0)}$ is simply the restriction to $\ell_p(X^{(0)})$ of the differential 
operator $d$
defined for every $f : X^{(0)} \to \hr$ by
$$\forall a=(a_- , a_+) \in X^{(1)}, \; df(a) = f(a_+) - f(a_-).$$  
The \emph{$k$-th $\ell_p $-cohomology group} of $X$ is 
$$\ell_p H^k(X) = \ker d^{(k)}/\textrm{Im}\ d^{(k-1)}.$$
Since $X$ is simply connected, every $1$-cocycle on $X$ is the differential
of a unique function $f: X^{(0)} \to \hr$ 
up to an additive constant. Therefore in degree $1$ we get a canonical isomorphism 
$$\ell_p H^1(X ) \simeq
\{f: X^{(0)} \to \hr \; | \; df\in \ell_p(X^{(1)})\}/ \ell_p(X^{(0)} ) + \hr ,$$
where $\hr$ denotes the set of constant functions on $X^{(0)} $.
In the sequel we shall always represent $\ell_p H^1(X )$  via this isomorphism. 

Equipped with the semi-norm induced by the $\ell_p$-norm of $df$
the topological vector space $\ell_pH^1(X )$ is a quasi-isometric
invariant of $X$.   Moreover if $X$ satisfies a linear isoperimetric inequality,
then $\ell_p H^1(X )$ is a Banach space, and 
$\ell_p H^1(X)$ injects 
in  $\ell_q H^1(X )$ for $1 < p \le q < +\infty $.
See \cite{Pa1}, \cite{G}, \cite{B4} for a proof of these results. 
 
The \emph{continuous first $\ell_p $-cohomology group} of $X$ is
$$\ell_p  H^1 _{\textrm{cont}} (X) := \{[f] \in \ell_p H^1 (X)  \; | \; f  \; \textrm{extends~continuously~
to}\;X^{(0)} \cup \partial X\},$$
where  $X^{(0)} \cup \partial X$ is the metric compactification of $X^{(0)}$ 
(see Section \ref{preliminaries} for the definition).  

Following P. Pansu \cite{Pa1} we introduce the following quasi-isometric numerical invariant of $X$:
$$p _{\neq 0}(X) = \inf \{p \ge 1 ~ | ~ \ell_p H^1 _{\textrm{cont}}(X) \neq 0\}.$$  

\subsection*{$\ell_p$-Equivalence relations} 
For $[f] \in \ell_p H^1 _{\textrm{cont}}(X)$ denote by $f_\infty : \partial X \to \mathbb R$ its 
boundary extension.
Following  M. Gromov (\cite[p.259]{G}, see also \cite{E}, \cite{B2})
we set
$$A_p (\partial X) :=
\{u : \partial X \to \hr  \;  | \;
 u = f_\infty \; \textrm{with}  \;
[f] \in \ell_p H^1 _{\textrm{cont}}(X) \},$$
and we define the \emph{$\ell_p $-equivalence relation} on $\partial X$  by ~:
$$z_1 \sim_p z_2 \Longleftrightarrow \forall u \in A_p(\partial X),
\ u(z_1) = u(z_2).$$
This is a closed equivalence relation on $\partial X$ 
which is invariant by the boundary extensions of the quasi-isometries
of $X$. Its cosets are called the \emph{$\ell _p$-cosets}. We define:
$$p_{\textrm{sep}} (X) = \inf \{p \ge 1 ~ | ~ A_p(\partial X) \; \textrm{separates~points~in}\;\partial X \}.$$
Equivalently $p_{\textrm{sep}} (X)$ is the infimal $p$ such that all $\ell_p$-cosets are points. 

Recall (from the introduction) that $X$ is \emph{nondegenerate}, if 
every $x \in X$ lies within uniformly bounded distance from all three sides of 
some ideal geodesic triangle. For nondegenerate spaces $X$, every class
$[f] \in \ell_p H^1 _{\textrm{cont}}(X)$ is fully determined by its boundary value $f_\infty$  
(\cite{St}, see also \cite{Pa1}, \cite{BP2} Th.3.1). 
More precisely $[f] = 0$ if and only if $f_\infty$
is constant. In particular we get:

\begin{proposition}\label{critique-p} Suppose $X$ is nondegenerate. Then
$$p_{\neq 0}(X)= \inf \{p \ge 1 ~ | ~ (\partial X/\! \sim _p )  ~\mathrm{is~not~a~singleton}\}.$$
\end{proposition}

We also notice that for nondegenerate spaces $X$, the $\ell_p$-cosets are
always connected in $\partial X$ (\cite{BourK} Prop. 10.1).

\subsection*{Relative $\ell_p$-cohomology}
\label{relative}
 Let $Y$ be a subcomplex of $X$. The \emph{$k$-th space of relative
$\ell_p$-cochains of $(X,Y)$} is 
$$C_p ^{(k)} (X,Y) :=  \{ \omega \in C_p ^{(k)} (X) \; | \; \omega \restr_{Y ^{(k)}} = 0 \}.$$ 
The \emph{$k$-th relative $\ell_p $-cohomology group} of $(X,Y)$ is 
$$\ell_p H^k(X,Y) = \frac{\ker \big(d^{(k)}: C_p ^{(k)} (X,Y) \to C_p ^{(k+1)} (X,Y)\big)}
{\textrm{Im}\big(d^{(k-1)} : C_p ^{(k-1)} (X,Y) \to C_p ^{(k)} (X,Y)\big)}.$$
A straightforward property is the following excision principle:
\begin{proposition}\label{excision} Suppose $U \subset Y$ is a subset such that
$Y \setminus U$ is a subcomplex of $Y$. Then for every $k \in \mathbb N$ the restriction map induces
a canonical isomorphism 
$$\ell_p H^k(X,Y) \simeq \ell_p H^k(X \setminus U,Y \setminus U).$$
\end{proposition}
Since $X$ is simply connected, by integrating every relative $1$-cocycle, 
we obtain the following canonical isomorphism:
\begin{eqnarray*}
\ell_p H^1 (X,Y) & \simeq & \{f: X^{(0)} \to \hr \; | \; df\in \ell_p(X^{(1)}) \; 
\textrm{and}\; f\restr _{E^{(0)}}  \; \textrm{is constant} \\ &  & \;\;\textrm{in every connected component}   \;
E \; \textrm{of} \; Y \}/\! \sim,
\end{eqnarray*}
where $f \sim g$ if and only if $f-g$ belongs to $\ell _p (X^{(0)}) + \hr$.
We will always represent $\ell_p H^1 (X,Y)$ via this isomorphism. 
Note that $\ell_p H^1 (X,Y)$ injects canonically in $\ell_p H^1 (X)$. 

We denote by 
$\ell_p H^1_{\textrm{cont}} (X,Y)$ the subspace of 
$\ell_p H^1 (X,Y)$ consisting of the classes $[f]$ 
such that $f$ extends continuously to $X^{(0)} \cup \partial X$.
We introduce two numerical invariants: 
\begin{itemize}
\item $p_{\neq 0}(X, Y)$ is the infimum of the $p \in [1, +\infty)$ such that 
the relative cohomology $\ell _p H ^1 _{\textrm{cont}}(X, Y)$ is nontrivial. 
\item $p_{\textrm{sep}}(X, Y)$ is the infimum of the $p \in [1, +\infty)$ such that
for every pair of distinct points $z_1,z_2 \in \partial X$, 
either $\{z_1,z_2\} \subseteq \partial E$ for some
connected component $E$ of $Y$, or
the relative cohomology $\ell _p H ^1 _{\textrm{cont}}(X, Y)$ separates
$\{z_1,z_2\}$.
\end{itemize}
In case there is no such $p$, 
we just declare the corresponding invariant to be equal to $+\infty$. We emphasize that the
property: $\partial E \cap \partial E' = \emptyset$ for every pair of distinct connected components
$E, E'$ of $Y$, is a 
necessary condition for the finiteness of
$p_{\textrm{sep}}(X, Y)$ .

Obviously one has 
$p_{\neq 0}(X) \le p_{\neq 0}(X, Y)$. 
When $X$ is nondegenerate, every element $[f] \in \ell _p H ^1 _{\textrm{cont}}(X, Y)$
is determined by its boundary extension $f_\infty$ ; and we get
$p_{\neq 0}(X, Y) \le p_{\textrm{sep}}(X, Y)$. 

\subsection*{A construction of cohomology classes}

One of the goals of the paper is to construct $\ell_p$-cohomology classes of $X$. 
We present here a construction that uses the relative cohomology of special subcomplexes of $X$.

\begin{definition}\label{defcut} 
A subcomplex $Y \subset X$ \emph{decomposes} $X$, if it is connected, 
simply connected and quasi-convex, 
and if its frontier $W := Y \setminus \inte(Y)$ enjoys the following properties:
\begin{enumerate}
\item For every pair $H_1, H_2$ of distinct connected components of $W$ one has
$\partial H_1 \cap \partial H_2 = \emptyset$.
\item  Every sequence $\{H_i\}_{i \in \mathbb N}$ of distinct connected components of $W$
subconverges in $X \cup \partial X$ to a singleton of $\partial X$.
\end{enumerate}
A collection $\{Y_j \}_{j \in J}$ of subcomplexes of $X$
\emph{fully decomposes} $X$, if it satisfies the following properties:
\begin{itemize} 
\item [(3)] Every subcomplex $Y \in \{Y_j \}_{j \in J}$ decomposes $X$. 
\item [(4)] For every pair of distinct points $z_1, z_2 \in \partial X$, there is a 
$Y \in \{Y_j \}_{j \in J}$,
such that for every connected component $E$ of $\overline{X \setminus Y}$ 
one has $\{z_1, z_2\} \nsubseteq \partial E$.
\end{itemize}
\end{definition}

The origin of Definition \ref{defcut} lies in group amalgams, as illustrated by the
following example. 

\begin{example} \label{excut} 
Let $A, B, C$ be three hyperbolic groups, suppose that $A$ and $B$ are non-elementary
and that $C$ is a proper quasi-convex malnormal subgroup of $A$ and $B$.
Then the amalgamated product $\Gamma := A \star _C B$ is a hyperbolic
group \cite{KapI}.  Let $K_A$, $K_B$, and $K_C$ be finite $2$-complexes with
respective fundamental groups $A$, $B$, and $C$, such that there are simplicial
embeddings $K_C\hookrightarrow K_A$, $K_C\hookrightarrow K_B$ inducing the
given embeddings of fundamental groups. We let $K$ be the result of 
gluing $K_A$ to $K_B$ along the copies of $K_C$, so $\pi_1(K)\simeq A \star_C B=\Gamma$.   
The universal cover $\tilde K$
is a union of copies of the universal covers
$\tilde K_A$ and $\tilde K_B$, whose incidence
graph is the Bass-Serre tree of the decomposition $\Gamma =A \star_C B$.
If we choose a copy $Y\subset \tilde K$ of $\tilde K_A$, then $Y$ decomposes $\tilde K$.
The frontier of $Y$ in $\tilde K$ breaks up into connected components which are stabilized by conjugates
of $C$.
\end{example}
  
The following proposition and corollary will serve repeatedly in the sequel. 

\begin{proposition}\label{propcut}
Suppose that a subcomplex $Y$ decomposes $X$ and let $W := Y \setminus \inte(Y)$
be its frontier in $X$. Then:
\begin{enumerate}
\item The restriction map
$\ell_p H^1_{\textrm{cont}} (X, \overline{X \setminus Y}) \to \ell_p H^1_{\textrm{cont}} (Y,W)$,
defined by $[f] \mapsto [f \restr _{Y^{(0)}}]$, 
is an isomorphism.
\item 
$p _{\neq 0} (X,\overline{X \setminus Y} ) 
=p _{\neq 0}(Y,W)$ \; and \; $p _{\textrm{sep}} (X,\overline{X \setminus Y} ) 
=p _{\textrm{sep}}(Y,W)$.
\end{enumerate}
\end{proposition}
 
\begin{corollary}\label{corcut1}
\begin{itemize}
\item [(1)] Suppose that a subcomplex $Y$ decomposes $X$, and let $W$
be its frontier. Then one has: 
$$p_{\neq 0}(X) \le p_{\neq 0}(Y, W).$$
\item [(2)] Suppose that a subcomplex collection $\{Y_j \}_{j \in J}$ fully decomposes
$X$, and let $W _j $ be the frontier of $Y_j$. Then one has: 
$$p_{\textrm{sep}} (X) \le \sup _{j \in J} p_{\textrm{sep}}(Y_j, W_j).$$ 
\end{itemize}
\end{corollary}

The proof of the proposition will use the

\begin{lemma}\label{lemcut}
Let $Y \subset X$ be a connected  subcomplex, let $W$ be its frontier in $X$,
and let $\mathcal E$ be the set of the connected components of $\overline{X \setminus Y}$.
Then:
\begin{enumerate}
\item The connected components of $W$ are precisely the subsets of the form
$E \cap Y$, with $E \in \mathcal E$. 
\item Their limit sets in $\partial X$ satisfy 
$\partial (E \cap Y) = \partial E \cap \partial Y$.
\item If $E \in \mathcal E$ is such that $\partial E \cap \partial Y = \emptyset$,
then $\partial E$ is open (and closed) in $\partial X$. 
\end{enumerate}
\end{lemma}

\begin{proof}[Proof of Lemma \ref{lemcut}]
(1). Notice that $W$ can be expressed as $W = (\overline{X \setminus Y}) \cap Y$. 
Hence it is enough to show 
that every subset $E \cap Y$ (with $E \in \mathcal E$) is contained in a connected component of $W$.
Consider the following exact
sequence associated to the ordinary simplicial homology of the couple $(X, \overline{X \setminus Y})$:
$$H_1(X) \to H_1 (X,\overline{X \setminus Y}) \stackrel{\overline{\partial}}{\ra} H_0 (\overline{X \setminus Y}).$$
Since $X$ is simply connected one has $H_1(X) =0$. By excising the open subset $X \setminus Y$ from $X$ and 
$\overline{X \setminus Y}$, we get a canonical isomorphism  
$H_1 (Y, W) \to H_1 (X,\overline{X \setminus Y})$. Therefore the boundary map
$\overline{\partial} : H_1 (Y, W) \to H_0 (\overline{X \setminus Y})$ is injective. 
Let $w_1, w_2$ be two distinct points in $E \cap Y^{(0)}$. Since $Y$ is connected there is path
$\gamma \subset Y^{(1)}$ joining $w_1$ to $w_2$. It defines a $(Y, W)$-relative $1$-cycle,
such that $\overline{\partial}([\gamma]) = [w_2 - w_1] = 0$ in $H_0 (\overline{X \setminus Y})$.
Since $\overline{\partial}$ is injective, $[\gamma]$ is trivial in $H_1 (Y, W)$.
Therefore there is a $2$-chain $\sigma$ in $Y$, such that $\partial \sigma - \gamma$ is a $1$-chain in $W$
whose boundary is $w_1 - w_2$. Hence $w_1$ and $w_2$ lie in the same connected component of $W$.

(2). It is enough to prove that 
$$(\partial E \cap \partial Y) \subset \partial (E \cap Y).$$ 
Let $z \in \partial E \cap \partial Y$,
and pick points $x_n \in E$, $y_n \in Y$ that both converge to $z$ as $n \to \infty$.
Consider the geodesic segment $[x_n, y_n]$; it must meet $E \cap Y$, let
$w_n$ be an intersection point. Then $w_n$ converges to $z$; thus $z \in \partial (E \cap Y)$.
 
(3). Let $E \in \mathcal E$ with 
$\partial E \cap \partial Y = \emptyset$. We claim that $\{\partial E , \partial (X \setminus E)\}$
is a partition of $\partial X$.
Clearly, one has $\partial X = \partial E \cup \partial (X \setminus E)$. 
Suppose by contradiction that 
$\partial E \cap \partial (X \setminus E)$ is nonempty. Let 
$z \in \partial E \cap \partial (X \setminus E)$, and pick $x_n \in E$ and $x_n ' \in X \setminus E$
that both converge to $z$. Consider the geodesic segment $[x_n , x_n ']$; it must meet $Y$, let
$y_n$ be an intersection point.
We get $z = \lim y_n \in \partial Y$, which is a contradiction. The claim follows.
\end{proof}

We can now give the 

\begin{proof}[Proof of Proposition \ref{propcut}] (1). We apply the excision property (Proposition \ref{excision}) to the complexes
$X$, $\overline{X \setminus Y}$ and to the subset $U = X \setminus Y$. 
Since 
$(\overline{X \setminus Y}) \setminus U = (\overline{X \setminus Y}) \cap Y = W$,
we obtain that $\ell_p H^1 (X, \overline{X \setminus Y})$ is isomorphic to $\ell_p H^1 (Y,W)$.
The complexes $X$ and $Y$ are simply connected, thus, via the canonical representation, 
the isomorphism can 
simply be written as 
$[f] \mapsto [f \restr _{Y^{(0)}}]$. 

We now suppose that $f \restr _{Y^{(0)}}$ extends continuously to
$Y^{(0)} \cup \partial Y$, and we prove that $f$ extends continuously to $X^{(0)} \cup \partial X$.   
Let $\mathcal E$ be the set of connected components of $\overline{X \setminus Y}$.
Consider a sequence $\{x_n\}_{n \in \hn} \subset X^{(0)}$ that converges in 
$X^{(0)} \cup \partial X$ to a point $x _\infty \in \partial X$.
We wish to prove that
$\{f(x_n)\}_{n \in \hn}$ is a convergent sequence. We distinguish several cases.

If $x _\infty$ belongs to $\partial X \setminus \partial Y$, then for $n$ and $m$ large enough the
geodesic segments $[x_n, x_m]$ do not intersect $Y$. Therefore 
they are all contained in the same component $E \in \mathcal E$.
Thus $f(x_n)$ is constant for $n$ large enough.

Suppose now that $x_\infty \in \partial Y$. Denote by $(f \restr _{Y^{(0)}}) _\infty$ the boundary extension of 
$f \restr _{Y^{(0)}}$ to $\partial Y$. We will prove that $f (x_n) \to 
(f \restr _{Y^{(0)}}) _\infty (x_\infty)$. By taking a subsequence if necessary, 
it is enough to consider the following special
cases: 
\begin{itemize}
\item [(A)]For every $n \in \mathbb N$, $x_n$ belongs to  $Y$.
\item [(B)]There is an $E \in \mathcal E$ such that for every $n \in \mathbb N$, $x_n$ belongs to  $E$.
\item [(C)]There is a sequence 
$\{E_n\}_{n \in \hn}$ of distinct elements of $\mathcal E$ such that $x_n \in E_n$.
\end{itemize}  
Case (A) is obvious. In case (B), let $y_n \in Y ^{(0)}$ which converges to $x_\infty$.
When one travels from $x_n$ to $y_n$ along a geodesic path 
$\gamma _n \subset X^{(1)}$, one meets
for the first time $Y ^{(0)}$ at a point $y'_n \in (E \cap Y^{(0)})$.
Since $y' _n \in \gamma _n$, it satisfies  $y'_n \to x_\infty$
when $n \to \infty$. Therefore we obtain : 
$$f(x_n) = f(y'_n) = f \restr _{Y^{(0)}}(y'_n) \to (f \restr _{Y^{(0)}})_\infty (x_\infty).$$
In case (C), let $y$ be an origin in $Y$. Every geodesic segment of $X$ joining $y$ to
$E_n$ meets $E_n \cap Y$.  
It follows from property (2) in Definition \ref{defcut}
and from Lemma \ref{lemcut}(1), that 
the sequence $\{E_n\}_{n \in \hn}$ converges in $X \cup \partial X$ to the singleton $\{x_\infty \}$.
For $n \in \hn$, pick $y_n \in E_n \cap Y^{(0)}$.  
The sequence $\{y_n \}_{n \in \hn}$
tends to $x_\infty$; this leads to 
$$f(x_n) = f \restr _{Y^{(0)}} (y_n) \to (f \restr _{Y^{(0)}})_\infty (x_\infty).$$

(2). The equality $p _{\neq 0} (X,\overline{X \setminus Y} ) 
=p _{\neq 0}(Y,W)$ follows trivially from part (1). We now establish the equality
$p _{\textrm{sep}} (X,\overline{X \setminus Y} ) = p _{\textrm{sep}}(Y,W)$.

The equality $p _{\textrm{sep}}(Y,W) \le p _{\textrm{sep}} (X,\overline{X \setminus Y} )$
is a straightforward consequence of Lemma \ref{lemcut}(2) and part (1) above. 

To establish the converse inequality, suppose that $\ell_p H^1_{\textrm{cont}} (Y,W)$
separates every pair of distinct points $z'_1, z'_2 \in \partial Y$ such that
$\{z'_1, z'_2\} \nsubseteq \partial H$ for every component $H$ of $W$.
Given a pair of distinct points $z_1, z_2 \in \partial X$, with $\{z_1, z_2\} \nsubseteq \partial E$ 
for every component $E \in \mathcal E$, we are looking for an element 
$[f] \in \ell_p H^1_{\textrm{cont}} (X, \overline{X \setminus Y})$
such that $f_\infty (z_1) \neq f_\infty (z_2)$. 

First, suppose there is an $E \in \mathcal E$ such that $z_1 \in \partial E$ and 
$\partial E \cap \partial Y = \emptyset$. Then, by Lemma \ref{lemcut}(3), the subset $\partial E$
is open and closed in $\partial X$. As an easy consequence the characteristic function $u$ of $\partial E$
writes $u = f_\infty$, with
$[f] \in \ell_q H^1_{\textrm{cont}} (X, \overline{X \setminus Y})$, 
for every $q \ge1$.

Suppose on the contrary, that neither $z_1$ nor $z_2$ belong to a limit set $\partial E$ 
with $\partial E \cap \partial Y = \emptyset$. For $i = 1,2$, define a point $z'_i \in \partial Y$
as follows : 
\begin{itemize}
\item If $z_i \in \partial Y$ set $z'_i := z_i$. 
\item If $z_i \notin \partial Y$, then there is an $E_i \in \mathcal E$ with $z_i \in \partial E_i$;
pick $z'_i \in \partial E_i \cap \partial Y$.
\end{itemize}
 From Lemma \ref{lemcut}(2) and from the assumption (1) in Definition \ref{defcut}, 
we have 
$\partial E_1 \cap \partial E_2 \cap Y = \emptyset$ 
for every distinct components $E_1, E_2 \in \mathcal E$. It follows that  
$z'_1 \neq z'_2$, and that $\{z'_1, z'_2\} \nsubseteq \partial H$ 
for every component $H$ of $W$. Our separability assumption 
on $\ell_p H^1_{\textrm{cont}} (Y,W)$, in combination with part (1) isomorphism, 
yields a desired element  
$[f] \in \ell_p H^1_{\textrm{cont}} (X, \overline{X \setminus Y})$.
\end{proof}

\begin{proof}[Proof of Corollary \ref{corcut1}]

(1).  From a standard inequality and Proposition \ref{propcut}(2), one has 
$$p_{\neq 0}(X) \le p_{\neq 0}(X,\overline{X \setminus Y}) = p_{\neq 0}(Y, W).$$ 

(2). From Definition \ref{defcut} and Proposition \ref{propcut}(2), one has
$$p_{\textrm{sep}} (X) \le \sup _{j \in J} p_{\textrm{sep}} (X,\overline{X \setminus Y_j}) 
= \sup _{j \in J} p_{\textrm{sep}}(Y_j, W_j)\,.$$
\end{proof}

\subsection*{Cohomology and the geometry of the boundary}\label{cohoboundary}

 The following result relates the $\ell_p$-cohomology of $X$ 
with the structure of the boundary $\partial X$, 
more precisely with the Ahlfors regular conformal dimension 
and the Combinatorial Loewner Property (see Section \ref{preliminaries} for the 
definitions).  It will serve as a main tool in the paper.

\begin{theorem}\label{Confdim}
Assume that $X$ is non-degenerate and that
$\partial X$ is connected and approximately self-similar, let $p \ge 1$.
Then : 
\begin{itemize}
\item [(1)] $p > \Confdim (\partial X)$ if and only if $(\partial X /\! \sim _p ) = \partial X$; 
in particular $$p_{\textrm{sep}} (X) = \Confdim (\partial X).$$
\item [(2)] If $\partial X$ satisfies the CLP, then for $1 \le p \le \Confdim (\partial X)$ 
the quotient  $\partial X/\! \sim _p$ is a singleton; in particular  $$p_{\neq 0}(X)
= p_{\textrm{sep}} (X) = \Confdim (\partial X).$$
\end{itemize}
\end{theorem}
\begin{proof} It follows immediately from Theorem \ref{modconf} in combination with 
Cor.10.5 in \cite{BourK} and Proposition \ref{critique-p}.
\end{proof}

\section{A qualitative bound for $p_{\textrm{sep}}(Y, W)$}\label{qualitative}

In this section we make a qualitative connection between 
the relative invariant $p_{\textrm{sep}}(Y,W)$
defined in Section \ref{relative}
and certain geometric properties of the pair $(Y,W)$
(see Corollary \ref{corqualitative}). This relies on the following result which is also of
independent interest. In the statement $\Delta (\cdot, \cdot)$ denotes the relative
distance, see (\ref{reldist}) for the definition.

\begin{theorem}
\label{thm-holderfunction}
For every $\al\in (0,1)$ there is a $D \ge 1$ with the following property.  
Suppose
$Z$ is a  bounded metric space, 
and $\C$ is a countable collection of closed positive diameter subsets of $Z$   
where $\Delta(C_1,C_2) \ge D$
for all $C_1,C_2\in \C$, $C_1\neq C_2$.   
Then, for every pair of distinct points $z_1, z_2 \in Z$, either
 $\{z_1, z_2\} \subseteq C$ for some $C \in \mathcal C$, or
there is a H\"older function $u\in C^\al(Z)$, 
such that :
\begin{enumerate}
\item  $u\restr_C$ is 
constant for every $C\in \C$.
\item If $C_1,C_2\in\C$ and  $u(C_1)= u(C_2)$, then $C_1=C_2$.
\item $u(z_1) \neq u(z_2)$. 
\end{enumerate}
\end{theorem} 

\smallskip

\begin{remark*}

1) The countability of $\C$ is only used to obtain (2); one gets plenty of functions 
without countability.


2) The argument given in the proof of Prop.1.3 of \cite{B2}
shows that there is a function $\alpha : (0, \infty) \to (0,1)$
with the following property. Let $Z$ be bounded metric space, 
and let $\C$ be a countable collection of closed positive diameter subsets of $Z$ 
with $\Delta(C_1,C_2) \ge D >0$
for all $C_1,C_2\in \C$, $C_1\neq C_2$. Then for $\alpha = \alpha(D)$ there is a H\"older 
function $u\in C^\al(Z)$ which satisfies conditions (1),(2),(3) above.
Theorem \ref{thm-holderfunction} above asserts that we can choose the function $\alpha$ so that
$\alpha(D) \to 1$ when $D \to \infty$.

\end{remark*}

\smallskip

To prove Theorem \ref{thm-holderfunction},  we consider a bounded metric space $Z$, 
and a 
countable collection $\mathcal C$ of closed positive diameter subsets of $Z$,   
such that $\Delta(C_1,C_2) \ge D$
for all $C_1,C_2\in \C$, $C_1\neq C_2$; here $D$ is a constant that is subject
to several lower bounds during the course of the proof.  
We can assume that $\diam(Z) \le 1$ 
by rescaling the metric of $Z$ if necessary. 
We will assume  
that $D > 8$, and pick  $\Lambda \in [4,\frac{D}{2})$ (we will need $\Lambda \ge 4$ 
in the proof of Sublemma \ref{lem-lipschitzbound}).   
Let $r_k=2^{-k}$, and let $\C_k=\{C\in \C ~ | ~ \diam(C)\in (r_{k+1},r_k]\}$.
Since $\diam(Z) \le 1$ we have $\mathcal C = \cup _{k \ge 0} \;\mathcal C _k$.
Given  a Lipschitz function $v_{-1} : Z \to \mathbb R$, we will construct 
the H\"older function $u$ as a convergent series $\sum_{j=-1}^\infty v_j$ where for every 
$k \ge 0$ : 
\begin{itemize}
\item $v_k$ is Lipschitz.
\item $v_k$ is supported in $\cup_{C\in \C_k}\;N_{\Lambda r_k}(C)$.
\item For every $0\le j\leq k$ and  $C\in \C_j$, the partial sum $u_k = \sum_{j=-1} ^k v_j$ is 
constant on $C$.
\end{itemize}
Thus one may think of $v_k$ as a ``correction'' which adjusts $u_{k-1}$ so that it becomes constant
on elements of $\C_k$.

\begin{lemma}
\label{lem-supportseparation}
Suppose  $0 \le i,j<k$, $i\neq j$, and there are $C_1\in \C_i$, $C_2\in \C_j$ such that
$$
\dist(N_{\Lambda r_i}(C_1),
N_{\Lambda r_j}(C_2))\leq r_k\,.
$$
Then 
 $|i-j|\geq \log_2(\frac{D}{6\Lambda})\,.$
\end{lemma}
\begin{proof}
We may assume $i<j$.
We have
$$
D\leq\Delta(C_1,C_2)\leq \frac{\dist(C_1,C_2)}{\left(\frac{r_j}{2}\right)}
\implies \dist(C_1,C_2)\geq \frac{Dr_j}{2}
$$
and
$$
\dist(C_1,C_2)\leq \Lambda r_i+\Lambda r_j+r_k\leq 3\Lambda r_i
$$
so
$$
\frac{r_i}{r_j}\geq \frac{D}{6\Lambda}\implies j-i\geq \log_2\left(\frac{D}{6\Lambda}\right)\,.
$$
\end{proof}

Let $n$ be the integer part of 
$\log_2(\frac{D}{6\Lambda})$; 
we will assume that $n\geq 1$. 

Let
$\{L_j\}_{j\in\Z}\subset [0,\infty)$ be an increasing sequence such that $L_j=0$ for
all $j\leq -2$, and let $\hat L_k=\sum_{j=1}^\infty L_{k-jn}$ (this is a finite sum since
$L_j=0$ for $j\leq -2$).

\begin{definition}
The sequence $\{L_j\}$ is \emph{feasible} if $L_k\geq  \hat L_k$
for all $k\geq 0$.
\end{definition}

As an example let $L_j=e^{\la j}$ for $j\geq -1$, and $L_j=0$ for $j\leq -2$.
Then $\{L_j\}$ is feasible if $e^{-\la n}<\frac12$.  In particular, we may take $\la$ small
when $n$ is large.

\begin{lemma}
Suppose $\{L_j\}$ is feasible, and $v_{-1}:Z\ra \R$ is $L_{-1}$-Lipschitz.
Then there is a sequence $\{v_k\}$, where for every $k \ge 0$\,{\rm :}
\begin{enumerate}
\item  $v_k$ is $L_k$-Lipschitz.
\item $\spt(v_k)\subset N_{\Lambda r_k}(\cup_{C\in\C_k}C)$.
\item $\|v_k\|_{C^0}\leq 2L_kr_k$.
\item For every $0 \le j\leq k$, and every $C\in \C_j$, the partial sum
 $u_k=\sum_{j = -1} ^k v_j$ is constant on $C$ .
\end{enumerate}
\end{lemma}
\begin{proof}

Assume inductively that for some $k\geq 0$, there exist functions
$v_{-1},\ldots,v_{k-1}$ satisfying the conditions of the lemma.
For every $C\in \C_k$, we would like to specify the constant value of the function
$u_k$.  To that end, choose a point $p_C\in C$ and some
$u_C\in [u_{k-1}(p_C)-\hat L_kr_k,u_{k-1}(p_C)+\hat L_kr_k]$.
Let 
$$
W_k=\big(Z\setminus (\cup_{C\in \C_k}N_{\Lambda r_k}(C))\big)
\sqcup\big(\cup_{C\in \C_k}C\big)\,,
$$ 
and
define $\bar v_k:W_k\ra \R$ by 
$$
\bar v_k(x)=
\begin{cases}
0 &\quad x\in Z\setminus \cup_{C\in \C_k}N_{\Lambda r_k}(C)\\
u_C-u_{k-1}(x) & \quad x\in  C\in \C_k\;.
\end{cases}
$$
Since $\Lambda < \frac{D}{2}$, the subset $W_k$ contains every 
$C \in \mathcal C_i$ with $i <k$.
Thus $\bar v_k(x)=0$ for $x \in C \in \mathcal C_i$ and for $i <k$.

\begin{sublemma}
\label{lem-lipschitzbound}
One has : 
$$
\Lip(\bar v_k)\leq \sum_{j=1}^\infty L_{k-jn}=\hat L_k\,.
$$

\end{sublemma}
\begin{proof}[Proof of the sublemma.]
Pick $x,y\in W_k$.

{\em Case 1. Suppose $x,y\in C$ for some $C\in \C_k$.}
Then
\begin{align*}
\frac{|\bar v_k(x)-\bar v_k(y)|}{d(x,y)}&=\frac{|u_{k-1}(x)-u_{k-1}(y)|}{d(x,y)}\\
&\leq\sum\{\Lip(v_j)\mid j<k,\;\spt(v_j)\cap C\neq\emptyset\}\\
&\leq \sum\{L_j\mid j<k,\;\spt(v_j)\cap C\neq \emptyset\}\\
&\leq \sum\{L_j\mid j<k,\;\exists C_1\in \C_j\;\text{s.t.}\;N_{\Lambda r_j}(C_1)\cap C\neq \emptyset\}\\
&\leq \sum_{j\ge 1} L_{k-jn}=\hat L_k\,,
\end{align*}
since the sequence $L_j$ is increasing, and consecutive elements of 
the set
$\{j ~ | ~ j<k,\;\exists C_1\in \C_j\;\text{s.t.}\;N_{\Lambda r_j}(C_1)\cap C\neq \emptyset\}$
differ by at least $n$, by Lemma \ref{lem-supportseparation}.

{\em Case 2. There is a $C\in \C_k$ such that $x\in C$, and $y\notin C$.}
Then $d(x,y)\geq \Lambda r_k$.   Reasoning as in Case 1, we have  $|u_{k-1}(p_C)-u_{k-1}(x)|
\leq \hat L_kr_k$.  Hence 
\begin{equation*}\label{Case2}
|\bar v_k(x)|=|u_C-u_{k-1}(x)|\leq |u_{k-1}(p_C)-u_{k-1}(x)|+\hat L_kr_k\leq 2\hat L_kr_k\,.
\end{equation*}
Therefore
$$
|\bar v_k(x)-\bar v_k(y)|\leq |\bar v_k(x)|+|\bar v_k(y)|
\leq 4\hat L_kr_k,
$$
so, since $\Lambda \ge 4$, 
$$
\frac{|\bar v_k(x)-\bar v_k(y)|}{d(x,y)}\leq \frac{4}{\Lambda}\hat L_k
\leq \hat L_k\,.
$$
Thus the sublemma holds.
\end{proof}

By McShane's extension lemma (see \cite{He}), there is an $\hat L_k$-Lipschitz
extension  $v_k:Z\ra \R$ of $\bar v_k$, where $\|v_k\|_{C^0}\leq \|\bar v_k\|_{C^0}\leq
2\hat L_kr_k$.  Since $\hat L_k\leq L_k$ by the feasibility assumption,
$v_k$ is $L_k$-Lipschitz and $\|v_k\|_{C^0}\leq 2 L_kr_k$.   
Moreover, by construction, $u_k$ is constant equal to $u_C$ on every
$C \in \mathcal C _k$, and is equal to $u_{k-1}$ on every $C \in \mathcal C_i$ with $i<k$.
Therefore condition (4) is satisfied and the lemma holds by induction.
\end{proof}

\begin{proof}[Proof of Theorem \ref{thm-holderfunction}.]
To get the H\"older bound, we use the feasible sequence $L_j = e^{\la j}$, and keep in mind that
we may take $\la$ close to $0$ provided $D$ is large.
Suppose  $x,y\in Z$ and $d(x,y)\in [r_{k+1},r_k]$. 
Then
\begin{align*}
|u_{k-1}(x)-u_{k-1}(y)|&
\leq r_k\sum\{L_j\mid j<k,\spt(v_j)\cap\{x,y\}\neq\emptyset\}\\
&\leq r_k\hat L_k\leq L_kr_k\,.
\end{align*}
Also,
\begin{align*}
\arrowvert \sum_{j\geq k}v_j(x)-\sum_{j\geq k}v_j(y)\arrowvert
&\leq 2\sum_{j\geq k}\|v_j\|_{C^0}\leq 4\sum_{j\geq k}L_jr_j\\
&=4\sum_{j\geq k}e^{\la j}2^{-j}
\leq 9 L_kr_k
\end{align*}
when $\la $ is small.  So
$$
\frac{|u(x)-u(y)|}{d(x,y)^\al}\leq \frac{10L_kr_k}{\left(\frac{r_k}{2}\right)^\al}
=10\cdot 2^\al\cdot \left(e^\la\cdot 2^{\al-1}\right)^k
$$
which is bounded independent of $k$ when $\la $ is small.

It is clear from the construction that if $\C$ is countable, then we may arrange that $u$ takes
different values on different $C$'s. Similarly, for a given pair of distinct points
$z_1, z_2 \in Z$, we can perform the above construction so that $u(z_1) \neq u(z_2)$. 
\end{proof} 

 We now give an application to hyperbolic spaces. Again we consider $Y$  
a hyperbolic non-degenerate simply connected
metric simplicial complex, with links of uniformly bounded complexity, and all simplices isometric to
regular Euclidean simplices with unit length edges.

\begin{corollary}\label{corqualitative} Let $Y$ be as above.
For every $\alpha \in (0,1)$, $C \ge 0$, there is a $D \ge 1$
with the following properties.
Suppose $W\subset Y$ is a subcomplex of $Y$
such that
\begin{enumerate}
\item Every connected component of $W$ is  $C$-quasiconvex in $Y$.
\item For every $y \in W$ there is a
complete geodesic $\ga\subset W$ lying in the same connected component of 
$W$, such that $\dist(y, \ga) \le C$.
\item The distance between distinct components of $W$ is at least $D$.
\end{enumerate} 
Then $p_{\textrm{sep}}(Y,W) \le \frac{1}{\alpha}\Confdim (\partial Y)$. 
\end{corollary}

The proof relies on Theorem \ref{thm-holderfunction} and on the following elementary lemma.
 
\begin{lemma}\label{lemmaqualitative}
Let $Y$ be a hyperbolic space and let $d$ be a visual metric on $\partial Y$.
For every $C \ge 0$ there are constants  $A > 0$, $B\ge 0$
such that for every $C$-quasiconvex subsets $H_1, H_2 \subset Y$ 
with non empty limit sets, one has
$$\Delta (\partial H_1, \partial H_2) \ge A \cdot a ^{\frac{1}{2} \dist(H_1,H_2)} - B,$$
where $a$ is the exponential parameter of $d$ (see (\ref{visual})). 
Moreover $A,B$ depend only on $C$, the hyperbolicity constant of $Y$, 
and the constants of the visual metric $d$.
\end{lemma}

\begin{proof}[Proof of the lemma]  
We denote the distance in $Y$ by $\vert y_1 - y_2 \vert$.
Pick $h_1 \in H_1, h_2 \in H_2$ with $\vert h_1 - h_2 \vert \le \dist(H_1,H_2)+1$. 
If $\dist(H_1,H_2)$
is large enough -- compared with $C$ and the hyperbolicity constant of $Y$ -- 
then for every $y_1 \in H_1$, $y_2 \in H_2$
the subset $[y_1, h_1] \cup [h_1, h_2] \cup [h_2, y_2]$ is a quasi-geodesic segment with controlled constants.
In particular the segment $[h_1, h_2]$ lies in a metric neighborhood  $N_r([y_1, y_2])$ 
where $r$ is controlled. 
Therefore $H_1 \cup [h_1, h_2] \cup H_2$ is a $C'$-quasiconvex subset, where $C'$ depends 
only on $C$ and the hyperbolicity constant of $Y$. 

Pick an origin  $y_0 \in Y$ and let $p$ be a nearest point projection of $y_0$ on 
$H_1 \cup [h_1, h_2] \cup H_2 $. We distinguish two cases.

\emph{Case 1 : $p \in [h_1, h_2]$}. Set $L = \max_{i=1,2} \dist (p, H_i)$; one has
with the relations (\ref{visual}) :
$$\Delta (\partial H_1, \partial H_2) \gtrsim \frac{a ^{-\vert y_0 - p \vert}}
{a ^{-\vert y_0 - p \vert -L}} = a ^L \ge a ^{\frac{1}{2} \dist(H_1,H_2)}.$$ 

\emph{Case 2 : $p \in H_1 \cup H_2$}. Suppose for example that $p$ belongs to $H_1$. 
Since every geodesic joining $H_1$ to $H_2$ passes close by $h_1$ one has 
$$\dist (\partial H_1, \partial H_2) \gtrsim a ^{-\vert y_0 - h_1 \vert}.$$
In addition every  geodesic joining $y_0$ to $H_2$ passes close by $p$ and thus close
by $h_1$. Therefore :
$$\diam (\partial H_2) \lesssim  a ^{-\vert y_0 - h_1 \vert - \dist(H_1,H_2)},$$
and so $\Delta (\partial H_1, \partial H_2) \gtrsim a ^{\dist(H_1,H_2)}$. 
\end{proof}

\begin{proof}[Proof of Corollary \ref{corqualitative}]  Let $\alpha \in (0,1)$ and let 
$d$ be a visual metric on $\partial Y$. 
  From the definition of the Ahlfors regular conformal dimension, and because
$\frac{\alpha + 1}{2 \alpha} >1$, 
there is an Ahlfors $Q$-regular metric  $\delta$ on $\partial Y$, with
$Q \le \frac{\alpha + 1}{2 \alpha}\Confdim (\partial Y)$, that is \emph{quasi-M\"obius equivalent} to $d$ i.e.\!
the identity map $(\partial Y, d) \to (\partial Y, \delta)$ is a quasi-M\"obius homeomorphism. 
This last property implies that the relative distances associated to $d$ and $\delta$
are quantitatively related (see \cite{BK1} Lemma 3.2).
Thus, thanks to Theorem \ref{thm-holderfunction} and Lemma \ref{lemmaqualitative}, 
there is a constant $D \ge 1$
such that if $W \subset Y$ satisfies the conditions
(1), (2), (3) of the statement, and if $\mathcal H$ denotes the family of its connected components, 
then,  for every pair of dictinct points $z_1, z_2 \in \partial Y$ with
$\{z_1, z_2\} \nsubseteq \partial H$ for every $H \in \mathcal H$, 
there is a function $u \in C^{\frac{\alpha+1}{2}} (\partial Y, \delta)$
with the following properties
\begin{itemize}
\item  $u \restr _{\partial H}$ is 
constant for every $H \in \mathcal H$.
\item $u(z_1) \neq u(z_2)$.
\end{itemize}   
We wish to extend $u$ to a continuous function $f : Y^{(0)} \cup \partial Y \to \mathbb R$
that is constant on every $H \in \mathcal H$.
To do so pick an origin $y_0 \in Y$ and observe that the nondegeneracy property of $Y$ yields 
the existence of a constant
$R \ge 0$ such that for every $y \in Y$ there is a $z \in \partial Y$ with 
$\dist (y, [y_0, z)) \le R$.

Let $y \in Y^{(0)}$. If $y$ belongs to a $H \in \mathcal H$, set $f (y) = u(\partial H)$.
If not, pick a $z \in \partial Y$ such that $\dist (y, [y_0, z)) \le R$
and define $f (y) = u(z)$. 

We claim that $f$ is a continuous function of 
$Y^{(0)} \cup \partial Y$. Let $a>1$ be the exponential parameter of the visual metric $d$.
Our hypothesis (2) implies that there is a constant 
$R' \ge 0$ so that for every $y \in Y^{(0)}$ there is a  $z \in \partial Y$
with $\dist (y, [y_0, z)) \le R'$  and 
$f(y) = u(z)$. Let $y_1, y_2 \in Y^{(0)}$ with $y_1 \neq y_2$. By relation (\ref{visualbis}),
their mutual distance in $Y^{(0)} \cup \partial Y$
is comparable to $a^{-L}$, where $L = \dist(y_0, [y_1,y_2])$. The metrics $d$ and $\delta$ being
quasi-M\"obius equivalent they are H\"older equivalent (see \cite{He} Cor.\! 11.5).  
Thus $u \in C ^\beta (\partial Y, d)$ for
some $\beta >0$ ; and we obtain : 
\begin{equation}\label{equaqualitative} 
\vert f (y_1) - f (y_2) \vert = \vert u(z_1) - u(z_2) \vert
\lesssim a ^{-\beta \dist(y_0, (z_1,z_2))} \lesssim a ^{-\beta L}.
\end{equation}
To establish the last inequality one notices that $\dist(y_0, (z_1,z_2)) - L$
is bounded by below just in terms of $R'$ and the hyperbolicity constant of $Y$.
The claim follows.

We now claim that $df \in \ell_p(Y^{(1)})$ for $p>\frac{2Q}{\alpha +1}$. To see it,
observe that $u$ is a Lipschitz function of $(\partial Y, \delta ^{\frac{\alpha +1}{2}})$
and $\delta ^{\frac{\alpha +1}{2}}$ is an Ahlfors $\frac{2Q}{\alpha +1}$-regular metric
on $\partial Y$ which is quasi-M\"obius equivalent to a visual metric. Therefore
the claim follows from Elek's extension process \cite{E} (see \cite[Prop.\! 3.2]{B4}  for
more details).

Hence the function $f$ defines an element of $\ell_p H^1 _{\textrm {cont}} (Y,W)$ for
$p >\frac{2Q}{\alpha +1}$, which separates $z_1$ and $z_2$. Thus 
$$p _{\textrm{sep}} (Y,W) \le \frac{2Q}{\alpha +1} \le \frac{1}{\alpha} \Confdim (\partial Y).$$
 \end{proof}

\section{Applications to amalgamated products}\label{amalgam}

This section uses Section \ref{cohoboundary} and Corollary \ref{corqualitative} to
construct 
examples of group amalgams whose $\ell_p$-cohomology
has specified behavior (see Corollary \ref{cor2amalgam} especially).

 Let $A$ be a hyperbolic group and let $C \leqq A$ be a finitely presentable subgroup.
Let $K_A$ and $K_C$ be finite $2$-complexes with
respective fundamental groups $A$ and $C$, such that there is a  simplicial
embedding $K_A\hookrightarrow K_C$ inducing the
given embedding of fundamental groups. Denote by $\tilde K_A$ the universal cover
of $K_A$ and pick a copy $\tilde K_C \subset \tilde K_A$ of the universal cover 
of $K_C$. 
The spaces and invariants $\ell_p H^1_{\textrm{cont}}(Y,W)$,
$p_{\neq 0}(Y)$, $p_{\textrm{sep}}(Y)$, $p_{\neq 0}(Y, W)$, $p_{\textrm{sep}}(Y, W)$, associated to the
pair 
$$(Y, W) = (\tilde K_A, \cup _{a \in A} a \tilde K_C),$$
will be denoted simply 
by 
$\ell_p H^1_{\textrm{cont}}(A, C)$,
$p_{\neq 0}(A)$, $p_{\textrm{sep}}(A)$, $p_{\neq 0}(A, C)$ and $p_{\textrm{sep}}(A, C)$.

We notice that $p_{\textrm{sep}}(A,C)$ is finite when $C$ is 
a quasi-convex malnormal subgroup 
(see Remark 2 after Theorem \ref{thm-holderfunction} and \cite{B2} for more details). 

Recall from \cite{KapI} that, given $A, B$ two hyperbolic groups and 
$C$ a quasi-convex malnormal subgroup of $A$ and $B$, the amalgamated product $A *_C B$
is hyperbolic.

\begin{proposition}\label{propamalgam}
Let $A, B, C$ three hyperbolic groups, suppose that $A$ and $B$ are non elementary
and that $C$ is a proper quasi-convex malnormal subgroup of $A$ and $B$.
Let $\Gamma$ be the amalgamated product $A *_C B$.
We have\,:
\begin{itemize}
\item [(1)] If $p_{\neq 0}(A,C) < p_{\textrm{sep}} (B)$ then for all $p \in (p_{\neq 0}(A,C), p_{\textrm{sep}} (B)]$
the $\ell _p$-equivalence relation on $\partial \Gamma$ possesses a coset different from a point
and the whole boundary $\partial \Gamma$.
\item [(2)] If $p_{\textrm{sep}}(A,C) < p_{\neq 0}(B)$ then for all $p \in (p_{\textrm{sep}}(A,C), p_{\neq 0}(B))$
the $\ell _p$-equivalence relation on $\partial \Gamma$ is of the following form : 
$$ z_1 \sim_p z_2 \Longleftrightarrow z_1 = z_2 ~~\textrm{or}~~\exists g \in \Gamma ~~\textrm{such that}~~
\{z_1,z_2\} \subset g (\partial B).$$
\end{itemize}
\end{proposition}

\begin{proof}[Proof of Proposition \ref{propamalgam}. ]
Part (2) is established in \cite{B2} Th.\! 0.1.
We provide here a more enlightening proof.

(1). Let $p \in (p_{\neq 0}(A,C),p_{\textrm{sep}} (B)]$ be as in the statement. From Example \ref{excut} and
Corollary \ref{corcut1}(1) we have $p>p_{\neq 0}(\Gamma)$. 
Hence, by Proposition \ref{critique-p}, the $\ell_p$-cosets are different from $\partial \Gamma$.
On the other hand we have the obvious inequality  $p \le p_{\textrm{sep}} (B) \le p_{\textrm{sep}} (\Gamma)$. 
Thus Theorem \ref{Confdim}(1)
shows that $\partial \Gamma$ admits a $\ell_p$-coset which is different from a singleton.

(2). Let $p \in (p_{\textrm{sep}}(A,C), p_{\neq 0}(B))$ be as in the statement. 
Let $u \in A_p (\partial \Gamma)$. 
Its restriction to every $g( \partial B)$ (with $g \in \Gamma)$ is constant since $p<p_{\neq 0}(B)$.
Thus every $g(\partial B)$ is contained in a $\ell_p$-coset.

Conversely to establish that 
$$ z_1 \sim_p z_2 \Longrightarrow z_1 = z_2 ~~\textrm{or}~~\exists g \in \Gamma ~~\textrm{such that}~~
\{z_1,z_2\} \subset g (\partial B),$$
we shall see that $A_p  (\partial \Gamma)$ separates the limit sets $g(\partial B)$.  
Consider finite simplicial complexes
$K_A, K_B, K_C, K$, with respective fundamental group $A, B, C, \Gamma$, as described in Example \ref{excut}.
Let $\tilde K$ be the universal cover of $K$ and choose copies $\tilde K_A , \tilde K_B \subset \tilde K$
of the universal covers of $K_A, K_B$.
We know from  Example \ref{excut} that $\tilde K_A$ decomposes $\tilde K$.

Pick two distinct subcomplexes $g \tilde K_B, g'\tilde K_B \subset \tilde K$ and a geodesic $\gamma \subset \tilde K$ 
joining them.
It passes through a subcomplex $g''\tilde K_A$.  Applying $(g'') ^{-1}$ if necessary, we may assume that $g'' =1$.
Therefore $g \tilde K_B$ and $g'\tilde K_B$ lie in different components of 
$\overline{\tilde K \setminus \tilde K_A}$. 
By Proposition \ref{propcut}(2), we have   $p_{\textrm{sep}}(A,C) = 
p_{\textrm{sep}} (\tilde K, \overline{\tilde K \setminus \tilde K_A})$. Since $p > p_{\textrm{sep}}(A,C)$,
we obtain that $A_p (\partial \Gamma)$ separates 
$g (\partial B)$ from $g' (\partial B)$.
\end{proof}

 From Theorem \ref{Confdim}(2), and since $\sim_p$ is invariant under the boundary extensions
of the quasi-isometries of $\Gamma$, we get :

\begin{corollary}\label{coramalgam}
Let $A, B, C, \Gamma$ be as in Proposition \ref{propamalgam}. 
\begin{itemize}
\item [(1)] If $p_{\neq 0}(A,C) < p_{\textrm{sep}} (B)$ then $\partial \Gamma$ does not admit the CLP.
\item [(2)] If $p_{\textrm{sep}}(A,C) < p_{\neq 0}(B)$ then any quasi-isometry of $\Gamma$ permutes the cosets $gB$,
for $g\in \Gamma$. More precisely, the image of a coset $gB$ by a quasi-isometry of $\Gamma$
lies within bounded distance (quantitatively) from a unique coset $g'B$.
\end{itemize}
\end{corollary}

In combination with Corollary \ref{corqualitative} this leads to: 
  
\begin{corollary}\label{cor2amalgam}
Let $A, B, C$ be as in Proposition \ref{propamalgam} and suppose that there is a decreasing sequence 
$\{A_n\}_{n \in \mathbb N}$ of finite index subgroups of $A$ such that $\cap _{n \in \mathbb N} A_n = C$.
Set $\Gamma \!_n := A_n * _C B$.
\begin{itemize}
\item [(1)] If $p_{\textrm{sep}} (A) < p_{\textrm{sep}} (B)$ then, for all $p \in (p_{\textrm{sep}} (A), p_{\textrm{sep}} (B)]$
and every $n$ large enough, the $\ell _p$-equivalence relation on $\partial \Gamma \!_n$ 
possesses a coset different from a point
and the whole $\partial \Gamma$. In particular for $n$ large enough, 
$\partial \Gamma \!_n$ does not admit the CLP.
\item [(2)] If $p_{\textrm{sep}} (A) < p_{\neq 0}(B)$ then, for $p \in (p_{\textrm{sep}}(A), p_{\neq 0}(B))$
and every $n$ large enough, the cosets of the $\ell_p$-equivalence relation on $\partial \Gamma \!_n$
are single points and the boundaries of cosets $gB$, for $g \in \Gamma \!_n$.
In particular, for large $n$, any quasi-isometry of $\Gamma \!_n$ permutes the cosets $gB$, for
$g\in \Gamma \!_n$.
\end{itemize}
\end{corollary}

\begin{proof}  According to Proposition \ref{propamalgam} and Corollary \ref{coramalgam} it is enough to prove that 
\begin{equation}\label{equamalgam}
\mathrm{limsup}_n~ p_{\textrm{sep}}(A_n, C) \le p_{\textrm{sep}} (A).
\end{equation}
We consider the simplicial complexes $\tilde K_A, \tilde K_C$ introduced in the beginning of the section,
and we set for $n \in \mathbb N$ :
$$(Y, W_n) := (\tilde K_A, \cup _{a \in A_n} a \tilde K_C).$$
Then just by definition we have\,: $p_{\textrm{sep}}(A_n, C) = p_{\textrm{sep}}(Y,W_n)$. 
On the other hand, 
since $CA_n C= A_n$, one has in the group $A$\;:
$$\inf\{\dist(C, aC) ~\vert~ aC \in A_n /C, ~ aC \neq C\} = \dist(1, A_n \setminus C).$$ 
The group $A$
is quasi-isometric to the space $\tilde K_A$. Therefore the last equality, in combination with
the hypothesis $\cap _{n \in \mathbb N} A_n = C$, implies that the minimal pairwise separation between
distinct components of $W_n$ tends infinity as $n \to \infty$.  
It follows from Corollary \ref{corqualitative} that 
$$\mathrm{limsup}_n ~p_{\textrm{sep}}(Y, W_n) \le \Confdim (\partial A).$$
With Theorem \ref{Confdim}(1) we get the desired inequality (\ref{equamalgam}). 
\end{proof}

\medskip

\begin{example}
\label{ex-marios_question}
 It is now possible to answer M. Bonk's question about examples of approximately 
self-similar Sierpinski carpets without the CLP.
 Pick two hyperbolic groups $A$ and $B$ whose boundaries are homeomorphic to the 
Sierpinski carpet, and which admit an isomorphic peripheral subgroup $C$. Assume
that $\Confdim (\partial A) < \Confdim (\partial B)$ and that there is a sequence  
$\{A_n\}_{n \in \mathbb N}$ of finite index subgroups of $A$ such that $\cap _{n \in \mathbb N} A_n = C$.
Such examples can be found among hyperbolic Coxeter groups, because quasi-convex subgroups of
Coxeter groups are separable. 
See Example \ref{extetrahedron}
and \cite{HW}.
Then for $n$ large enough the boundary of $A_n * _C B$ is homeomorphic to the Sierpinski
carpet. Moreover, according to Corollary \ref{cor2amalgam}, it doesn't admit the CLP for $n$
large enough.
\end{example} 

\medskip

\begin{example}
Let $M$, $M'$ and $N$ be closed hyperbolic (\emph{i.e.} \!constant curvature $-1$)
manifolds with $1 \le \dim (N) < \dim (M) < \dim (M')$. Suppose that $M$ and $M'$ contain
as a submanifold an isometric totally geodesic copy of $N$. 
The group $\pi _1 (N)$ is separable in $\pi _1 (M)$ (see \cite{Ber}). 
For the standard hyperbolic space $\mathbb H ^k$ of dimension $k \ge 2$, one has by \cite{Pa1} :
$\Confdim (\partial \mathbb H ^k) = p_{\neq 0} (\mathbb H ^k) = k -1$. 
Therefore the assumptions
in items (1) and (2) of Corollary \ref{cor2amalgam} are satisfied with
$A = \pi_1(M)$, $B=\pi_1(M')$ and $C=\pi_1(N)$. It follows that the 
manifold $M$ admits a finite cover $M_n$ containing an isometric totally geodesic copy of $N$,
such that the space $K := M_n \sqcup _N M'$ possesses the following properties : 
\begin{itemize}
\item For $p \in (\dim M -1, \dim M' -1)$, by letting $\tilde K$ be the universal cover of $K$, 
the cosets of the $\ell_p$-equivalence relation 
on $\partial \tilde K$
are points and the boundaries of lifts of $M'$.
In particular $\partial \tilde K$ doesn't satisfy
the CLP.
\item Every quasi-isometry of $\tilde K$ permutes the lifts of $M' \subset K$.
\end{itemize}
The second property may also be proven using the topology of the boundary, or using coarse topology.
By combining topological arguments with the zooming method of R. Schwartz \cite{Sch}, 
one can deduce from the second property
that every quasi-isometry of $\tilde K$ lies within bounded distance from an isometry.
\end{example}

\section{Elementary polygonal complexes}\label{elementary}

In this section we compute the invariants $p_{\neq 0}(Y,W)$ and $p_{\textrm{sep}}(Y,W)$
in some very special cases. They will serve in the next sections to obtain  upper bounds
for the conformal dimension and for the invariant $p_{\neq 0}(X)$.

\begin{definition} A \emph{polygonal complex} is a connected simply connected
$2$-cell complex $X$ of the following form :
\begin{itemize}
\item Every $2$-cell is isomorphic to a polygon with  at least $3$ sides.
\item Every pair of $2$-cells shares at most a vertex or an edge.
\end{itemize} 
The number of sides of a $2$-cell is called its \emph{perimeter},
the number of $2$-cells containing an edge is called its \emph{thickness}. 
\end{definition}

\begin{definition} An \emph{elementary polygonal complex} is a polygonal
complex $Y$ whose edges are colored black or white,  that enjoys the following properties :
\begin{itemize}
\item Every $2$-cell has even perimeter 
at least $6$.
\item The edges on the boundary 
of every $2$-cell are alternately black and white. 
\item Every white edge has thickness $1$, and every black
edge has thickness at least $2$.
\end{itemize}
\noindent
We will  equip every elementary polygonal complex $Y$ 
with a length metric of negative curvature, by
identifying every $2$-cell with a constant negative curvature right angled
regular polygon of unit length edges; in particular $Y$ is a Gromov
hyperbolic metric space (quasi-isometric to a tree). 
The union of the white edges of $Y$  is called its \emph{frontier}.  The frontier is a locally
convex subcomplex of $Y$, and hence every connected component
is a $CAT(0)$ space, i.e.\! subtree of $Y$.   We call such components   \emph{frontier trees} of $Y$.

The (unique) elementary polygonal complex, whose $2$-cells have constant perimeter $2m$
and whose black edges have constant thickness $k$ ($m \ge 3, k \ge 2$), will be 
denoted by $Y_{m,k}$.
\end{definition}

We denote by $\mathcal T$ the 
frontier of $Y$, and we consider the associated invariants :
$p _{\neq 0}(Y, \mathcal T), p_{\textrm{sep}}(Y, \mathcal T)$. When $Y = Y_{m,k}$ we 
write $\mathcal T _{m,k}$, $p_{m,k}$, $q_{m,k}$ for brevity.
 
\begin{theorem}\label{theopq}
For $m\ge 3$ and $k \ge 2$ one has :
$$p_{m,k} =  q_{m,k} = 1 + \frac{\log (k-1)}{\log (m-1)}.$$
\end{theorem}

\begin{proof}
We first establish the upper bound 
\begin{eqnarray}\label{q}
q_{m,k} \le 1 + \frac{\log (k-1)}{\log (m-1)}.
\end{eqnarray}
For this purpose we will construct some elements in  
$\ell_p H^1_{\textrm{cont}} (Y_{m,k}, \mathcal T _{m,k})$. 
Set  $Y_m := Y_{m,2}$ for simplicity. Observe that $Y_m$ is a planar polygonal complex.

For any choice of $2$-cells $c \subset Y_{m,k}$, $d \subset Y_m $ there
is an obvious continuous polygonal map $r_{c,d} : Y_{m,k} \to Y_m$ sending $c$ to $d$ and such that
for every $y_0 \in c$ and every $y \in Y_{m,k}$ one has $\vert r_{c,d}(y_0) - r_{c,d}(y) \vert =
\vert y_0 - y \vert$.

Observe that the frontier of $Y_m$ is an union of disjoint geodesics.
We will define a continuous map $\varphi : Y_m \cup \partial Y_m \to \hh ^2 \cup \S^1$
which maps every frontier geodesic of $Y_m$ to an ideal point in $\S^1$.
To do so, fix a $2$-cell $d \subset  Y_m $. At first we define $\varphi \restr_d $ so that
its image is a regular ideal $m$-gon. In other words $\varphi$ collapses every white edge of $d$
to an ideal point in $\S^1$, and these ideal points are $m$ regularly distributed points 
in $\S^1$.
For $n \in \hn$, let $P _n$ be the union of the $2$-cells of $Y_m$ whose combinatorial
distance to $d$ is less than or equal to $n$.  By induction we define  $\varphi$ on
the subcomplex $P_n$, so that 
\begin{itemize}
\item [(i)] The images of the
$2$-cells of $Y_m$ form a tesselation of $\hh ^2$ by ideal $m$-gons.
\item [(ii)] The images of the frontier geodesics of $Y_m$, passing through the subcomplex
$P_n$, are $m(m-1)^{n-1}$ ideal points equally spaced on $\S^1$.
\end{itemize}

Let $u : \S^1 \to \hr$ be a Lipschitz function. Since $\varphi (Y_m ^{(0)}) \subset \S^1$,
the composition $u \circ \varphi \circ r_{c,d}$ is well defined on $Y_{m,k} ^{(0)}$.
Let $f : Y_{m,k} ^{(0)} \to \hr$ be this function. By construction its restriction to every
frontier tree $T \subset \mathcal T _{m,k}$ is constant. Moreover $f$ extends continuously to 
$Y_{m,k} ^{(0)} \cup \partial Y_{m,k}$ since $r_{c,d}$ does. It remains to estimate
the $p$-norm of $df$. Let $C$ be the Lipschitz constant of $u$. 
There are 
$m(m-1)^n(k-1)^n$ 
black edges at the frontier of the subcomplex
$r_{c,d} ^{-1} (P_n)$.  If $a$ is such an edge, one has from property (ii) above : 
$$\vert df (a) \vert \le \frac{2 \pi C}{m(m-1)^{n-1}}.$$  
Therefore :
\begin{eqnarray*}
\Vert df \Vert _p ^p \; \lesssim \; \sum _{n \in \hn} \frac{(m-1)^n(k-1)^n}{(m-1)^{pn}}
\; = \; \sum _{n \in \hn} \big( (m-1)^{1-p} (k-1) \big) ^n.
\end{eqnarray*}
Thus $[f]$ belongs to $\ell_p H^1_{\textrm{cont}} (Y_{m,k}, \mathcal T _{m,k})$
for $(m-1)^{1-p} (k-1) < 1$ \textit{i.e.} for 
$$p > 1 + \frac{\log (k-1)}{\log (m-1)}.$$
In addition, by varying the $2$-cell $c$ and the fonction $u$,
one obtains functions $f$ that separate any given pair of distinct points $z_1,z_2 \in \partial Y_{m,k}$
such that $\{z_1,z_2\} \nsubseteq \partial T$ for every $T \in \mathcal T _{m,k}$.
Inequality (\ref{q}) now follows.

To establish the theorem it remains to prove :
\begin{eqnarray}\label{p}
p_{m,k} \ge  1 + \frac{\log (k-1)}{\log (m-1)}.
\end{eqnarray}
To do so we consider, for $m, k, \ell \ge 3$, the polygonal complex
$\Delta _{m, k, \ell}$ defined by the following properties :
\begin{itemize}
\item Every $2$-cell has perimeter $2m$.
\item Every edge is colored black or white, and the edge
colors on the boundary of each $2$-cell are alternating.
\item The thickness of every black edge is $k$, while the
thickness of white edges is $\ell$.  
\item The link of every vertex is the full bipartite graph
with $k + \ell$ vertices.
\end{itemize}
As with elementary polygonal complexes, we metrize $\De_{m,k,\ell}$ so that
each cell is a regular right-angled hyperbolic polygon, and hence $\De_{m,k,\ell}$ 
is a right-angled Fuchsian building.
The union of its white edges is a disjoint
union of totally geodesic trees in $\Delta _{m, k, \ell}$.
By cutting $\Delta _{m, k, \ell}$ along these trees, one divides $\Delta _{m, k, \ell}$
into subcomplexes, each isometric to $Y _{m,k}$. 
Let $Y \subset \Delta _{m, k, \ell}$ be one such subcomplex.  It  
decomposes $\Delta _{m, k, \ell}$. Denote by $\mathcal T$ its frontier.
By Corollary  \ref{corcut1} one gets that
$p_{\neq 0}(\Delta _{m, k, \ell} ) \le p_{\neq 0}(Y, \mathcal T ) =p_{m,k}$.
On the other hand $\partial \Delta _{m, k, \ell} $
is known to admit the CLP (see \cite{BP1, BourK}). Hence, with 
Theorem \ref{Confdim}(2) 
we \mbox{obtain :}
\begin{eqnarray*} 
p_{m,k} \ge p_{\neq 0}(\Delta _{m, k, \ell} ) = \Confdim (\partial \Delta _{m, k, \ell} ).
\end{eqnarray*}
>From \cite{B1} formula (0.2), one has $\Confdim (\partial \Delta _{m, k, \ell} ) = 1+\frac{1}{x}$
where $x$ is the unique positive number which satisfies :
$$\frac{(k-1)^x + (\ell -1)^x}{\big(1 + (k-1)^x \big)\big(1+ (\ell -1)^x \big)} = \frac{1}{m}.$$
By an easy computation we get that
$$\lim _{\ell \to +\infty} \Confdim (\partial \Delta _{m, k, \ell} ) = 
1 + \frac{\log (k-1)}{\log (m-1)}.$$
Inequality \ref{p} follows. 
\end{proof}

\begin{corollary} \label{corpq} Let $Y$ be an elementary polygonal complex. 
Assume that the perimeter of every $2$-cell of $Y$ lies in $[2m_1, 2m_2]$
and that the thickness of every black edge lies in $[k_1, k_2]$,
with $m_1 \ge 3$ and $k_1 \ge 2$. Then : \;
$$1 + \frac{\log (k_1-1)}{\log (m_2-1)}\le p_{\neq 0}(Y, \mathcal T) \le p_{\textrm{sep}}(Y, \mathcal T) 
\le 1 + \frac{\log (k_2-1)}{\log (m_1-1)}.$$
\end{corollary}

\begin{proof} We first establish the last inequality. 
By adding $2$-cells to $Y$ if necessary, we obtain an elementary polygonal complex $Y'$, whose 
black edges are of constant thickness $k_2$, and whose $2$-cell perimeters are larger than or equal
to $2m_1$ . Let $\mathcal T '$ be its frontier.
The cellular embedding $(Y, \mathcal T) \to (Y', \mathcal T')$
induces a restriction map 
$\ell_p H^1 _{\textrm{cont}} (Y', \mathcal T ') \to \ell_p H^1 _{\textrm{cont}}(Y, \mathcal T )$,
which leads to 
$p_{\textrm{sep}} (Y, \mathcal T) \le p_{\textrm{sep}} (Y', \mathcal T')$.

We now compare $p_{\textrm{sep}} (Y', \mathcal T')$ with $q _{m_1 , k_2}$.
Observe that any polygon $P$ of perimeter larger than $2m_1$ contains a $2m_1$-gon whose black 
edges are contained in black edges of $P$, and whose white edges are contained in
the interior of $P$.  With this observation 
one can construct  embeddings
$\varphi : Y_{m_1 , k_2} \to Y'$ in such a way that   $\varphi (Y_{m_1 , k_2})$ 
decomposes $Y'$ and   Proposition \ref{propcut}  yields 
a monomorphism $$\ell_p H^1 _{\textrm{cont}}(Y _{m_1 , k_2} , \mathcal T _{m_1 , k_2} ) 
\hookrightarrow \ell_p H^1 _{\textrm{cont}}(Y', \mathcal T ').$$
For $p> q _{m_1 , k_2}$,  
by varying the embedding $\varphi$,
we see that $\ell_p H^1 _{\textrm{cont}}(Y', \mathcal T ')$
separates any given pair of distinct points $z_1,z_2 \in \partial Y'$
such that $\{z_1,z_2\} \nsubseteq \partial T'$ for every $T' \in \mathcal T'$.
Therefore
$p_{\textrm{sep}} (Y', \mathcal T') \le q _{m_1 , k_2}$, 
and the last inequality of Corollary \ref{corpq} follows from Theorem \ref{theopq}.

The first inequality can be proved in a similar way.
\end{proof}

\section{Applications to polygonal complexes}\label{polygonal}

Building on earlier Sections \ref{cohoboundary} and \ref{elementary}, this section derives some results
on the conformal dimension of polygonal complex boundaries (Proposition
\ref{propappli}).

Let $X$ be a polygonal complex. If the perimeter of every
$2$-cell is larger than or equal to $7$, then we identify every $2$-cell 
with a constant negative curvature 
regular polygon of unit length edges and of angles 
$\frac{2\pi}{3}$. The resulting length metric on $X$ is of negative curvature, in particular
$X$ is hyperbolic.

If  every
$2$-cell has perimeter at least  $5$, and the link of of every vertex contains
no circuit of length $3$, then we identify every $2$-cell with a constant negative curvature 
regular polygon of unit length edges and of angles $\frac{\pi}{2}$. Again the resulting length metric 
on $X$ is of negative curvature. 

The following result gives a upper bound for the conformal dimension of the boundary at infinity
of such polygonal complexes. In some cases a lower bound can be obtained
by M. Gromov's method of ``round trees'', see \cite{G} p.\! 207, \cite{B0}, \cite{Mac2}.

\begin{proposition}\label{propappli}
Let $X$ be a polygonal complex of negative curvature as above; 
suppose that $\partial X$ is connected and approximately
self-similar. 
\begin{enumerate}
\item If the perimeter of every $2$-cell is at least $n \ge 5$,
the thickness of every edge is at most $k \ge 2$, and the link of every vertex contains
no circuit of  length $3$, then 
$$\Confdim (\partial X) \le 1 + \frac{\log (k-1)}{\log (n-3)}.$$
\item If the perimeter of every $2$-cell is at least $n \ge 7$ and
the thickness of every edge is at most $k \ge 2$, then
$$\Confdim (\partial X) \le 1 + \frac{\log (k-1)}{\log (n-5)}.$$
\end{enumerate} 
\end{proposition}

\begin{proof} We shall construct a family of embedded elementary polygonal complexes
that fully decomposes $X$. The statement will then follow from Corollaries \ref{corcut1}(2),
\ref{corpq}
and Theorem \ref{Confdim}(1).

(1) First we associate to every edge $e \subset X$ a \emph{rooted directed tree} 
defined by the following process :
\begin{itemize}
\item[(A)] Join by a directed segment the middle of $e$ to the center $O$ of every $2$-cell $c$ containing $e$,
\item[(B)] Connect by a directed segment the center $O$ to the middle of every edge $e'$ of $c$ 
that is not adjacent nor equal to $e$,
\item[(C)] Restart the process with $e'$ and every $2$-cell distinct from $c$ that contains
$e'$.
\end{itemize}
Let $T_e$ be the resulting graph in $X$.  We claim that $T_e$ is a bi-Lipschitz
embedded tree in $X$. To do so, we first notice that from its definition 
every edge of $T_e$ admits a direction. 
A path in $T_e$ will be called a \emph{directed path} if its edges are all
directed in the same way. Pick a directed path $\gamma \subset T_e$.
Since $X$ is right angled and non-positively curved, 
item (B) in combination with angle considerations implies
that the union of the $2$-cells met by $\gamma$
is a convex subset of $X$.  Therefore  
all directed paths are uniformly bilipschitz
embedded in $X$. In addition, two different directed paths in $T_e$ 
issuing from the same point have distinct endpoints. Thus $T_e$ is an embedded
tree in $X$. Moreover every backtrack free path in $T_e$ is the union of at most two
directed paths making an angle bounded away from $0$. 
Since $X$ is non-positively curved, one obtains that all backtrack free paths in $T_e$ are
uniformly bilipschitz embedded. The claim follows now easily.

Pick an $0< \epsilon \le \frac{1}{4}$ and consider the $\epsilon$-neighborhood
$Y_e$ of $T_e$ in $X$. This is a bi-Lipschitz embedded elementary polygonal complex. The perimeter
of every of its $2$-cell is larger than or equal to $2(n-2)$, and the thickness of its edges
is smaller than or equal to $k$. The family $\{Y_e \}_{e \in X^{(1)}}$ fully decomposes
$X$.  Indeed, for a given pair of distinct points $z_1, z_2 \in \partial X$, 
let $c \in X^{(2)}$ be a $2$-cell such that the geodesic $(z_1, z_2)$ connects two
non-adjacent sides $e_1, e_2$ of $c$ (possibly $(z_1,z_2)$ intersects $c$ in a 
side $e_3$ of $c$ which joins $e_1$ to $e_2$). 
Let $e$ be a side of $c$ which is different from $e_1$ and
$e_2$.   Then $Y_e$ separates $z_1$ and $z_2$.
Therefore according to Theorem \ref{Confdim}(1), Corollaries \ref{corcut1}(2)
and \ref{corpq} we get 
$$\Confdim (\partial X) = p_{\textrm{sep}} (X) 
\le \sup _{e \in X^{(1)}} p_{\textrm{sep}}(Y_e, \mathcal T_e) \le 1 + \frac{\log (k-1)}{\log (n-3)},$$
where $\mathcal T_e$ denotes the frontier of $Y_e$.

(2). The method is the same apart from a slight modification in the construction process
of the rooted directed trees. Item (B) becomes :
\begin{itemize}
\item[(B')] Connect $O$ to the middle of every edge $e'$
of $c$ whose combinatorial distance to $e$ is at least $2$.
\end{itemize} 
We claim that $T_e$ is again a
bi-Lipschitz embedded tree in $X$. To see this, consider again a directed path $\gamma \subset T_e$ and
the $2$-cells $c_1, ... , c_n \subset X$ successively met by $\gamma$.  Their union is not convex
in $X$ but a slight modification is. Indeed pick $i \in \{1, ... ,n-1\}$, let $x$ be one of the vertices
of the segment $c_i \cap c_{i+1}$, and let $y \in c_i \setminus c_{i+1}$, $z \in  c_{i+1} \setminus c_i$
be the vertices adjacent to $x$. Denote by $\sigma _x $ the convex hull of the 
subset $\{x,y,z\} \subset X$. By considering the link at $x$, one sees that it either degenerates to the union
$[yx] \cup [xz]$, or it is a simplex contained in the unique $2$-cell which contains $x,y,z$.
Since $X$ is a non positively curved polygonal complex with unit length edges
and $\frac{2\pi}{3}$ angles, the angles of $\sigma _x$ at $y$ and $z$ are smaller
than $\frac{\pi}{6}$. Moreover item (B') implies that $y \notin c_{i-1}$ and $z \notin c_{i+2}$.
It follows from angle considerations that the union 
$$(\cup _i  c_i) \cup (\cup _x \sigma _x)$$
defines a locally convex isometric immersion of an abstract CAT($0$) space into $X$,
and it is therefore a global embedding with convex image. 

The rest of the proof is similar to case (1). 
Here the $2$-cells of the associated elementary polygonal complex
$Y_e$ have perimeter larger than or equal to $2(n-4)$ because of (B').
\end{proof}

\bigskip
\begin{example} 
\label{ex-johns_question}
Assume that the link of every vertex is the $1$-skeleton of the 
$k$-dimensional cube and that the perimeter of every $2$-cell is equal to $n \ge 5$.
Then Proposition \ref{propappli}(1) in combination
with the lower bound established in \cite{B0} p.140 gives
$$ 1 + \frac{\log (k-1)}{\log(n-3) + \log 15} \le \Confdim (\partial X) \le
1 + \frac{\log (k-1)}{\log(n-3)}.$$
When the link is the $1$-skeleton of the $3$-dimensional cube $\partial X$
is homeomorphic to the Sierpinski carpet.  For even $n$, examples of such complexes are provided
by Davis complexes of Coxeter groups; their boundaries admit the CLP (see \cite{BourK} 9.4). 
Therefore one obtains examples
of Sierpinski carpet boundaries satisfying the CLP, whose conformal dimension is arbitrarily close to $1$.
This answers a question of  J. Heinonen and J. Mackay. 
\end{example}

\bigskip
\begin{example} \label{exvisual} Assume now that the link of every vertex is the full bipartite graph with 
$k+k$ vertices and that the perimeter of every $2$-cell is equal to $n \ge 5$.
>From \cite{B1} one has 
$$\Confdim (\partial X) = 1 + \frac{\log (k-1)}{\log\big(\frac{n}{2}-1 + \sqrt{(\frac{n}{2}-1)^2 -1}\big)},$$
which is quite close to the upper bound obtained in Proposition \ref{propappli}. 
We emphasize that the Hausdorff dimensions of the visual metrics on $\partial X$ 
do not give in general such precise upper bounds.
For example it follows from \cite{LL} that if every $2$-cell of $X$ is isometric to the same right
angled polygon $P \subset \mathbb H^2$, then the Hausdorff dimension 
of the associated visual metric is larger
than 
$$1 + \frac{\mathrm{length}(\partial P)}{\mathrm{area}(P)} \log(k-1).$$
\end{example} 
 
\section{Applications to Coxeter groups}\label{sectioncox}

This section applies earlier results for the invariants 
$\Confdim (\partial \Gamma)$ and $p_{\neq 0}(\Gamma)$ to 
Coxeter groups $\Gamma$ (Corollary 
\ref{corcox} and Proposition \ref{propcox})

Recall that a group $\Gamma$ is a \emph{Coxeter group} if it admits
a presentation of the form
$$ \Gamma = \langle s , ~ s \in S ~ \vert ~ s ^2 =1, ~(s t)^{m_{st}} =1
~\textrm{for}~ s \neq t \rangle , $$
with $\vert S \vert < + \infty $, and  $m_{st}   \in \{2,3, ..., + \infty \} $.
To such a presentation one associates a finite simplicial graph $L$ whose vertices
are the elements $s \in S$ and whose edges join the pairs $(s,t)$ such that
$s \neq t$ and $m_{st} \neq +\infty$. 
We label every edge of $L$ with the corresponding integer $m_{st}$. 
The labelled graph $L$ is called the \emph{defining graph} of $\Gamma$. 
The valence of a vertex $s \in L^{(0)}$ will be denoted by $\val(s)$.

Using Proposition \ref{propappli} we will deduce :

\begin{corollary} \label{corcox} 
Suppose $\Gamma$ is a Coxeter group with defining graph $L$.
\begin{enumerate}
\item If for every $(s,t) \in L^{(1)}$ one has $\val(s) \le k$, 
$m_{st} \ge m \ge 3$ and 
$L$ contains no circuit of length $3$, then $\Gamma$ is hyperbolic and
$$\Confdim (\partial \Gamma) \le 1 + \frac{\log (k - 1)}{\log (2m - 3)}.$$
\item If for every $(s,t) \in L^{(1)}$ one has $\val(s) \le k$,
$m_{st} \ge m \ge 4$, then $\Gamma$ is hyperbolic and
$$\Confdim (\partial \Gamma) \le 1 + \frac{\log (k - 1)}{\log (2m - 5)}.$$ 
\end{enumerate}
\end{corollary}

The above corollary shows that  global bounds for the  valence 
and the integers $m_{st}$ yield upper bounds for   the conformal dimension.
In contrast, the following result asserts that  local bounds are enough to obtain 
upper bounds for  $p_{\neq 0}(\Gamma)$.

\begin{proposition}\label{propcox} Suppose $\Gamma$ is a hyperbolic Coxeter group
with defining graph $L$. 
For $s \in L^{(0)}$ set $m_s := \inf _{(s,t) \in L^{(1)}} m_{st}$.
\begin{enumerate}
\item Suppose that there is an $s \in L^{(0)}$ with $\val (s) \ge 2$, $m_s \ge 3$ which does not
belong to any length $3$ circuit of $L$. Then  
$$p_{\neq 0}(\Gamma) \le 1 + \frac{\log (\val (s) - 1)}{\log (m_s - 1)}.$$
\item Suppose that there is an $s \in L^{(0)}$ with $\val (s) \ge 2$, $m_s \ge 5$. Then 
$$p_{\neq 0}(\Gamma) \le 1 + \frac{\log (\val (s) - 1)}{\log (m_s - 3)}.$$
\end{enumerate}
\end{proposition}

\begin{remark*} The definition of the invariant $p_{\neq 0}(\Gamma)$ given in Section \ref{equivalence} 
requires $\Gamma$ to be hyperbolic. However the above proposition extends to non hyperbolic 
Coxeter groups as well.
In this case the conclusion is simply that $\ell _p H^1 (\Gamma) \neq 0$ for $p$ larger than
the right hand side.
\end{remark*}

The proofs of Corollary \ref{corcox} and Proposition \ref{propcox} rely on the fact that
every Coxeter group $\Gamma$ has a  
properly discontinuous, cocompact, isometric action on a $CAT(0)$
cellular complex $X$ called the Davis complex of $\Gamma$. We list below some of its properties
(see \cite{D} Chapters 7 and 12 for more details).

For $I \subset S$, denote by $\Gamma \!_I$ the subgroup of $\Gamma$ generated by $I$.
It is again a Coxeter group; its defining graph is the maximal subgraph of $L$ whose vertex set is $I$.
By attaching a simplex to every subset $I \subset S$ such that $\Gamma \!_I$ is finite, one obtains
a simplicial complex $\Sigma L$ whose $1$-skeleton is the graph $L$. The Davis complex $X$ enjoys the
following properties :
\begin{itemize}
\item  The  $1$-skeleton of $X$  identifies naturally with the Cayley graph of 
$(\Gamma,S)$; in particular every edge of $X$ is
labelled by a generator $s \in S$.
\item Every $k$-cell is isometric to a $k$-dimensional Euclidean polytope.
\item The link of every vertex is isomorphic to the simplicial complex $\Sigma L$.
\item For $(s,t) \in L^{(1)}$ the corresponding $2$-cells of $X$
are regular Euclidean polygons of perimeter $2m_{st}$ and unit length edges 
alternately labelled $s$ and $t$. More generally if $I \subset S$ spans a 
$k$-simplex in $\Sigma L$, then the corresponding
$k$-cells of $X$ are isometric to the Euclidean polytope which is 
the Davis complex of the finite Coxeter group $\Gamma \!_I$.
\item For every $I \subset S$, the Davis complex of $\Gamma \!_I$ has a canonical
isometric embedding in $X$.
\end{itemize}

\begin{proof}[Proof of Corollary \ref{corcox}] 
Given a free product $G = A \star B$ of two hyperbolic groups $A,B$, it is well known that 
$$\Confdim (\partial G) = \max \{\Confdim (\partial A), ~ \Confdim (\partial B)\}. $$
Thus, by decomposing $\Gamma$ as a free product, where each factor is a Coxeter group associated to
a connected component of $L$, we can and will restrict ourself to the case $L$ is connected. 

For $I \subset S$, the hypotheses in (1) or (2) imply that the subgroup $\Gamma \!_I$ is finite if and only if 
$\vert I \vert = 1 ~\mathrm{or}~ 2$. Therefore one has $\Sigma L =L$ and the
Davis' complex $X$ is a polygonal complex. It clearly satisfies the hypothesis (1) or (2) 
in Proposition \ref{propappli} with $n =2m$. 
The corollary follows.
\end{proof}

\begin{proof} [Proof of Proposition \ref{propcox}]

(1). Pick $s \in L^{(0)}$ as in the statement, set $J := \{t \in S ~;~ (s,t) \in L^{(1)}\}$ 
and consider the Coxeter subgroup $A \leqq \Gamma$ 
generated by $\{s\} \cup J$. Its Davis complex, denoted by $X_A$, isometrically
embeds in $X$.

Since there is no length $3$ circuit containing $s$, the defining graph of $A$ 
consists only of segments joining $s$ to its neighbour vertices
$t \in J$. Therefore $X _A$ is an elementary polygonal complex; its black edges
are those labelled by $s$ and the white ones are labelled by an element of $J$.
The perimeter of its $2$-cells is at least
 $2m_s \ge 6$ and the thickness of its black edges is equal to $\val(s)$.
Denote by $\mathcal T _A$ its frontier.

The absence of a length $3$ circuit containing $s$ implies that for every $t \in J$ and 
$K \subset S$ with $\{s,t\} \subsetneqq K$,
the subgroup $\Gamma \!_K$  is infinite. From the previous description of the Davis complex $X$
we obtain that  
$X_A \setminus \mathcal T _A$ is an open subset of $X$.
Therefore $X _A$ decomposes $X$, since $X$ is simply connected.
With Corollaries \ref{corcut1}(1) and \ref{corpq} we get
$$p_{\neq 0}(X) \le p_{\neq 0}(X _A, \mathcal T _A) \le 1 + \frac{\log (\val (s) - 1)}{\log (m_s - 1)}.$$ 

(2). Let $s$ and $J$ be as in case (1). Consider the union of the $2$-cells of $X$ that
correspond to the family of edges $(s,t) \in L^{(1)}$ with $t \in J$. Let $\Xi$ be one of its 
connected components. It is a $2$-cell complex whose universal cover is an elementary
polygonal complex. We color its edges  black and white : the black edges
are those labelled by $s$ and the white ones those labelled by an element of $J$.
The perimeter of its $2$-cells is larger
than or equal to $2m_s \ge 10$ and the thickness of its black edges is equal to $\val(s)$.

Since $m_s \ge 5$ we see that for every $t \in J$ and 
$K \subset S$ with $\{s,t\} \subsetneqq K$,
the subgroup $\Gamma \!_K$  is infinite. Thus $\Xi$ minus the union of its white edges
is an open subset of $X$. We will construct an elementary
polygonal complex $Y \subset \Xi$ that decomposes $X$. The statement will then follow from
Corollaries \ref{corcut1}(1) and \ref{corpq} again.

The construction of $Y$ uses a variant of the method presented in the proof of
Proposition \ref{propappli}. We associate to every black edge $e \subset \Xi$ 
a \emph{rooted directed tree} 
defined by the following process :
\begin{itemize}
\item [(A)] Join by a directed segment the middle of $e$ to the center $O$ of every $2$-cell $c$ containing $e$,
\item [(B)] Connect by a directed segment the center $O$ to the middle of every black edge $e'$ of $c$ 
that is distinct from $e$ and from the two nearest ones,
\item [(C)] Restart the process with $e'$ and every $2$-cell distinct from $c$ that contains
$e'$.
\end{itemize}
Let $T_e$ be the resulting graph in $\Xi$. We claim that $T_e$ is a bi-Lipschitz
embedded tree in $X$. To see this we first modify the constant curvature of every $2$-cell of $\Xi$
so that their angles become equal to $\frac{3 \pi}{4}$. Since the original angles were
at least equal to $\pi - \frac{\pi}{m_s} \ge \frac{3 \pi}{4}$ and since 
$\frac{3 \pi}{4} + \frac{3 \pi}{4}+ \frac{\pi}{2} = 2 \pi$, the complex $X$ remains nonpositively
curved. 

Consider a directed path $\gamma \subset T_e$ as defined in the proof of Proposition \ref{propappli},
and let $c_1, ... , c_n \subset \Xi$ be
the $2$-cells successively met by $\gamma$.  Their union is not convex 
in $X$, but a slight modification is. Indeed, pick $i \in \{1, ... ,n-1\}$, let $x$ be one of the vertices
of the segment $c_i \cap c_{i+1}$, and let $y \in c_i \setminus c_{i+1}$, $z \in  c_{i+1} \setminus c_i$
be the vertices adjacent to $x$. Consider the geodesic simplex $\sigma _x \subset X$ whose
vertices are $x,y,z$. Since $X$ is non positively curved with unit length edges
and $\frac{3\pi}{4}$ angles, the angles of $\sigma _x$ at $y$ and $z$ are smaller
than $\frac{\pi}{4}$. Moreover item (B) implies that for two consecutive
segments $c_{i-1} \cap c_i$ and $c_i \cap c_{i+1}$ the associated simplices are disjoint
(they are separated by a black edge of $c_i$ at least). 
It follows from angle considerations that 
$$(\cup _i  c_i) \cup (\cup _x \sigma _x)$$
is a convex subset of $X$. 
The rest of the argument and of the construction is similar to the proof 
of Proposition \ref{propappli}. 
Here the $2$-cells of the associated elementary polygonal complex
$Y$ have perimeter larger than or equal to $2(m_s - 2)$ because of the second item of the construction
process.
\end{proof}

\medskip
\begin{example} \label{extetrahedron}
When a hyperbolic group boundary $\partial \Gamma$ satisfies the CLP one 
knows from Theorem \ref{Confdim}(2) that $p_{\neq 0}(\Gamma) = \Confdim (\partial \Gamma)$.
So Proposition \ref{propcox} yields an upper bound for the conformal dimension of 
CLP Coxeter boundaries.
For example consider the following hyperbolic Coxeter group 
$$ \Gamma = \langle s_1,\ldots,s_4\mid s_i^2=1,\;(s_is_j)^{m_{ij}}=1
\;\mbox{for}\;i\neq j\rangle\,,$$
where
the order $m_{ij}$ is finite for all $i\neq j$ and 
$\sum _{i \neq j} \frac{1}{m_{ij} } <  1$
for all $j \in \{1, ..., 4 \}$. The associated 
graph is the $1$-skeleton of the tetrahedron. 
The visual boundary is homeomorphic 
to the Sierpinski carpet, so its conformal dimension is larger than $1$ \cite{Mac}.
Moreover it admits the CLP \cite{BourK}. 
Define 
$$m = \max _{1 \le i \le 4} (~\min _{j \neq i} ~m_{ij} ).$$
If $m \ge 5$ then from Proposition \ref{propcox}(2) we get that
$$\Confdim (\partial \Gamma) \le 1 + \frac{\log 2}{\log (m - 3)}.$$ 
In particular if we choose an $i \in \{1, ..., 4 \}$, fix the orders
$\{m_{jk}\}_{j \neq i, k \neq i}$, and let
the $\{m_{ij}\}_{j \neq i}$ go to $+\infty$, the conformal dimension tends to $1$.
We obtain in such a way a family of Coxeter groups with Sierpinski carpet boundaries 
of different conformal dimensions, which all contain an isomorphic peripheral
subgroup, namely the subgroup generated by $\{s_j\}_{j \neq i}$. Existence of groups
with these properties was evoked in Example \ref{ex-marios_question}.
\end{example} 

\medskip
\begin{example} The previous examples are quite special because they
satisfy the CLP independently of the choice of the coefficients $m_{st}$.
In general the CLP for Coxeter group boundaries is very sensitive to the coefficients $m_{st}$.
The reason is the following: suppose that the graph of a hyperbolic
Coxeter group $\Gamma$ decomposes as $L = L_1 \cup L_2$
where $L_i$ ($i=1,2)$ is a \emph{flag} subgraph of $L$ (\emph{i.e.}\! every edge
of $L$ with both endpoints in $L_i$ belongs to $L_i$). Denote by $\Gamma \!_i$ the Coxeter
group with defining graph $L_i$, and assume that :
\begin{itemize}
\item There is a vertex $s \in L_1$ 
such that no edges issuing from $s$ belong to $L_2$.
\item The Coxeter group $\Gamma \!_2$ satisfies 
$\Confdim (\partial \Gamma \!_2) > 1$.
\end{itemize}
Then if the coefficients of the edges issuing from $s$ are large enough compared
to $\Confdim (\partial \Gamma\!_2)$,
we get from Proposition \ref{propcox}(2) and  Theorem \ref{Confdim}(1) :
$$p_{\neq 0}(\Gamma) < \Confdim (\partial \Gamma \!_2) \le \Confdim (\partial \Gamma),$$
and so, according to Theorem \ref{Confdim}(2), the CLP fails for $\partial \Gamma$. 

Concrete examples are provided, for instance, by Coxeter groups whose defining graphs are complete bipartite.
Let $L(k,\ell)$ be the full bipartite graph with $k$ black vertices and $\ell$ 
white vertices ($k\ge 3$, $\ell \ge 3$). 
When all the coefficients
$m_{st}$ are equal to the same integer $m \ge 3$ then it is known that the corresponding  
Coxeter group boundary admits the CLP \cite{BP1, BourK}. 
Now suppose  that the number $k$ of black vertices is at least equal to $4$.
Pick one black vertex and decompose $L(k,\ell)$ accordingly :  
$$L(k,\ell) = L(1, \ell) \cup L(k -1 ,\ell).$$
Next, choose the coefficients $m_{st} \ge 3$ arbitrarily on the edges of the second factor graph and 
let the coefficients of the first one
go to $+\infty$. Since we have $k-1 \ge 3$, the boundary of the second factor group is 
homeomorphic to the Menger sponge; thus its conformal dimension is larger than $1$ \cite{Mac}.
Therefore the above discussion applies and the CLP fails.
\end{example}


\bibliographystyle{alpha}
\bibliography{rellph1}

\noindent Marc Bourdon, Universit\'e Lille 1, D\'epartement de math\'ematiques, Bat. M2, 59655 Villeneuve d'Ascq,
France, bourdon@math.univ-lille1.fr 

\medskip

\noindent Bruce Kleiner, Courant Institute, New York University,
251 Mercer Street,
New York, N.Y. 10012-1185, USA, 
bkleiner@courant.nyu.edu 

\end{document}